\documentclass[a4paper]{lms}

\usepackage{amsmath, amssymb, amsrefs, enumerate, ifpdf, mathdots, stmaryrd, verbatim}
\usepackage{xypic}
\xyoption{curve}

\ifpdf
  \usepackage[pdftex]{graphicx}
\else
  \usepackage[dvips]{graphicx}
\fi

\title[Algebraic theories in homotopy theory]{Algebraic theories, span diagrams and commutative monoids in homotopy theory}
\author{James D. Cranch}

\classno{55P48 (primary), 55U40, 18C10, 18D05, 18G30 (secondary)}

\newtheorem{thm}{Theorem}[section]
\newtheorem{prop}[thm]{Proposition}

\newtheorem{corol}[thm]{Corollary}

\newnumbered{example}[thm]{Example}
\newnumbered{remark}[thm]{Remark}
\newnumbered{claim}[thm]{Claim}
\newnumbered{defn}[thm]{Definition}

\newnumbered{innerp}{Proof}
\newenvironment{innerproof}[1]%
  {\begin{proof*}[#1]}
  {\hfill\checkmark%
   \end{proof*}}

\newcommand{\dar}{\ar@{-->}}
\newcommand{\dl}{\ar@{--}}
\newcommand{\Ar} {\ar@{=>}}
\newcommand{\Dar}{\ar@{==>}}
\newcommand{\iar}{\ar@{^{(}->}}
\newcommand{\sear}{\ar@{-{>>}}}
\newcommand{\rar}{\ar@{~>}}

\DeclareMathOperator{\Ob}{Ob}

\DeclareMathOperator{\id}{id}
\DeclareMathOperator{\im}{im}
\DeclareMathOperator{\Fun}{Fun}
\DeclareMathOperator{\Map}{Map}
\DeclareMathOperator{\Mod}{Mod}

\DeclareMathOperator{\Mon}{Mon}

\DeclareMathOperator{\Hom}{Hom}

\DeclareMathOperator{\Arr}{Arr}
\DeclareMathOperator{\Spaces}{Spaces}

\DeclareMathOperator{\colim}{colim}
\DeclareMathOperator{\Theories}{Theories}

\DeclareMathOperator{\Aut}{Aut}
\DeclareMathOperator{\Alg}{Alg}
\DeclareMathOperator{\Free}{Free}

\newcommand{\isom}{\operatorname{\cong}}
\newcommand{\fib}{\mathrm{fib}}
\newcommand{\N}{\mathbb{N}}

\newcommand{\cA}{\mathcal{A}}
\newcommand{\cB}{\mathcal{B}}
\newcommand{\cC}{\mathcal{C}}
\renewcommand{\cD}{\mathcal{D}}
\newcommand{\cE}{\mathcal{E}}

\newcommand{\cJ}{\mathcal{J}}
\newcommand{\cK}{\mathcal{K}}
\renewcommand{\cL}{\mathcal{L}}

\newcommand{\cO}{\mathcal{O}}

\newcommand{\cU}{\mathcal{U}}
\newcommand{\fC}{\mathfrak{C}}
\newcommand{\Fin}{{\mathrm{FinSet}}}

\newcommand{\Fins}{{\Fin_*}}

\newcommand{\Finop}{{\Fin^\op}}
\newcommand{\Finsop}{{\Fin_*^\op}}
\newcommand{\FSI}{{\Fin^{\isom}}}

\newcommand{\Set}{\mathrm{Set}}

\newcommand{\sSet}{\mathrm{sSet}}
\newcommand{\sCat}{\mathrm{sCat}}
\newcommand{\Cat}{\mathrm{Cat}}
\newcommand{\Cinfty}{\Cat_\infty}
\newcommand{\Cinftypp}{\Cat_\infty^{\mathrm{pp}}}

\newcommand{\TwoSpan}{2\Span}
\newcommand{\Span}{\mathrm{Span}}

\newcommand{\op}{\mathrm{op}}

\newcommand{\Funpp}{\Fun^{\mathrm{pp}}}

\newcommand{\timeso}[1]{\mathop{\times}_{#1}}
\newcommand{\Spant}{\Span^\times}
\newcommand{\tcCt}{\tilde{\cC}^\times}
\newcommand{\cCt}{\cC^\times}
\newcommand{\cDt}{\cD^\times}
\newcommand{\cCot}{\cC^\otimes}
\newcommand{\Algt}{\Alg^\times}
\newcommand{\Algot}{\Alg^\otimes}
\newcommand{\bAlg}{\overline{\Alg}}
\newcommand{\Mont}{\Mon^\times}
\newcommand{\Monot}{\Mon^\otimes}
\newcommand{\bMon}{\overline{\Mon}}

\newcommand{\tp}{\tilde{p}}
\newcommand{\Spanc}{\Span^{\textrm{coll}}}
\newcommand{\Finsc}{\Fins^{\textrm{coll}}}
\newcommand{\limit}{\varprojlim}
\newcommand{\Monoids}{\mathrm{Mon}}
\newcommand{\coh}{\textrm{coh}}

\newcommand{\Ho}{\operatorname{ho}}

\newcommand{\skel}{\mathrm{skel}}

\newcommand{\cCD}{\cC^{\Delta^1}}

\newcommand{\Th}[1]{\mathrm{Th}_{#1}}
\newcommand{\ThS}[1]{\Th{#1}^\Delta}
\newcommand{\ThMon}{\Th{\mathrm{Mon}}}

\newdir{pb}{:(1,-1)@^{|-}}
\def\pb#1{\save[]+<12 pt,0 pt>:a(#1)\ar@{pb{}}[]\restore}

\begin{document}
\begin{abstract}
We adapt the notion of an algebraic theory to work in the setting of quasicategories developed recently by Joyal and Lurie. We develop the general theory at some length.

We study one extended example in detail: the theory of commutative monoids (which turns out to be essentially just a $2$-category). This gives a straightforward, combinatorially explicit, and instructive notion of a commutative monoid. We prove that this definition is equivalent (in appropriate senses) both to the classical concept of an $E_\infty$-monoid and to Lurie's concept of a commutative algebra object.
\end{abstract}
\maketitle 
\tableofcontents

\section{Introduction}

\subsection{The theory of monoids}

Let $M$ be a commutative monoid; we might be interested in natural operations $M^a\rightarrow M^b$. Here is an example of a natural operation $M^3\rightarrow M^4$:
\[(a,b,c)\longmapsto(b+a,c,0,a+a+a+b+c).\]

Such an operation consists of adding up some copies of the things we started with. We can regard this as a two-stage process: first we make copies, then we add. So we can factor this operation as
\[(a,b,c)\longmapsto(b,a,c,a,a,a,b,c)\longmapsto(b+a,c,0,a+a+a+b+c).\]

In general we can associate natural operations to maps of finite sets:
\begin{itemize}
\item Given a map $f:X\leftarrow U$ of sets, we can produce a \emph{copying} map $\Delta_f:M^X\rightarrow M^U$ via $(\Delta_fA)_u = A_{f(u)}$.
\item Given a map $g:U\rightarrow Y$ of sets, we cap produce an \emph{addition} map $\Sigma_g:M^U\rightarrow M^Y$ via $(\Sigma_gA)_y = \sum_{g(u)=y}A_u$.
\end{itemize}
Of course, we can compose these, and so given any diagram of finite sets
\[X\stackrel{f}{\longleftarrow}U\stackrel{g}{\longrightarrow}Y\]
we get an operation $\Sigma_g\circ\Delta_f$, which sends $M^X\rightarrow M^Y$ via \[(\Sigma_g\circ\Delta_f)(A)_y = \sum_{g(u)=y}A_{f(u)}.\]

We refer to a diagram of sets with this shape as a \emph{span diagram}. It certainly seems natural to suggest that span diagrams should give all the natural operations on a commutative monoid. However, span diagrams $X\leftarrow U\rightarrow Y$ and $X\leftarrow U'\rightarrow Y$ yield identical operations if they are isomorphic in the sense that
\begin{displaymath}
\xymatrix@R-20pt{&U\ar[dr]\ar[dd]^\wr&\\
X\ar[ur]\ar[dr]&&Y\\
&Y'\ar[ur]&}
\end{displaymath}
commutes.

Moreover, we can compose span diagrams: we use pullbacks:
\begin{displaymath}
\left(\vcenter{\xymatrix@R-8pt@C-16pt{&U\ar[dl]\ar[dr]&\\X&&Y}}\right) \circ
\left(\vcenter{\xymatrix@R-8pt@C-16pt{&V\ar[dl]\ar[dr]&\\Y&&Z}}\right) = 
\left(\vcenter{\xymatrix@R-8pt@C-16pt{&U\timeso{Y}V\ar[dl]\ar[dr]&\\X&&Z}}\right).
\end{displaymath}
It can quickly be checked that this is the right thing to do, using the formula above for $\Sigma_g\circ\Delta_f$.

These span diagrams thus form a category $\ThMon$. Hence it is reasonable to believe that this encodes the structure of commutative monoids precisely.

In fact, there is a classical result that it does: a commutative monoid is the same thing as a product-preserving functor from $\ThMon$ to sets.

The work of Lawvere generalises this point of view considerably; given a category $T$ generated under taking finite products by a single object, we say that $T$ is an \emph{algebraic theory} and that a product-preserving functor from $T$ to sets is a \emph{model of $T$}.

We might aim to apply this to homotopy theory: Badzioch \cite{Bad1} has shown that the models of $\ThMon$ in $\Spaces$ are exactly the generalised Eilenberg--Mac Lane spaces. His paper works in ordinary category theory, but a note observes that all results carry through in the world of simplicial categories.

We might hope for a different theory: frequently, a more useful notion of commutative monoid in $\Spaces$ is the notion of $E_\infty$-monoid \cite{Adams-ILS}. This is a monoid which is commutative only up to coherent homotopy. So we might ask, how might we change $\ThMon$ in order to get this more nuanced theory?

To find an answer, we must realise that we lose valuable information when we pass to isomorphism classes of span diagrams to form the category $\ThMon$.

\subsection{Theories in quasicategories}

In order to preserve this information, we use \emph{quasicategories}, as developed by Joyal \cite{Joyal} and Lurie \cite{HTT}. A quasicategory is a simplicial set obeying some extra axioms: these axioms are slightly less stringent than the axioms for a Kan complex.

The philosophy is as follows: a category can be regarded as a simplicial set, via the standard nerve construction. But an $\infty$-groupoid (a higher categorical object consisting of objects, invertible maps, and all higher homotopies between them) can also be regarded as a Kan complex (consisting of points, 1-simplices and higher simplices), and hence also inhabits the world of simplicial sets.

A quasicategory is a generalisation of both. In particular, it has:
\begin{itemize}
\item objects, represented by $0$-simplices;
\item morphisms, represented by $1$-simplices, which need not be invertible (as in an ordinary category);
\item homotopies, represented by higher simplices, which are all invertible in some appropriate sense.
\end{itemize}
Because of the nature of this theory, we shall use the terms ``0-simplex'', ``point'' and ``object'' interchangeably, and similarly also use the terms ``1-simplex'', ``edge'' and ``morphism'' interchangeably (depending perhaps on whether we are thinking of the object as a simplicial set, as a model for a space, or as a generalised category).

Likewise, we find ourselves using ``functor'' and ``map''  interchangeably, when we do not wish to be precise about whether our objects or quasicategories or more general simplicial sets.

We start with some preliminaries (Section \ref{2-cats-quasicats}), giving general results on quasicategories, and on their relation to lower-dimensional category theory.

We then (in Section \ref{theories-generalities}) build a quasicategory $\Span$ which is the appropriate version of $\ThMon$ for $E_\infty$-objects in a quasicategory. As in the category $\ThMon$, objects are finite sets, and morphisms are span diagrams. However, we build in higher simplices, which consist of more elaborate span-like diagrams in finite sets.

The quasicategory $\Span$ turns out to be equivalent to a $2$-category: all the cells of degree $3$ or more are effectively uninteresting. This greatly facilitates technical checks of its properties.

After setting up $\Span$ and its basic properties, we spend some time proving that it does exactly what we want: we compare it to the classical theory of $E_\infty$-algebras and also to Lurie's theory of commutative monoid objects.

\subsection{Further applicability}

It is worth mentioning some of the potential consequences of this approach to universal algebra in quasicategories that are not realised in this paper.

One perfectly general comment is that the approach is, in some senses, slightly more flexible than any operad-theoretic approach (as in \cite{HA}). In particular, an operad can only describe structures subject to axioms which do not mention any particular variable twice in the same formula: the distributive law $(a+b)c = ac+bc$ is not accessible by this approach, since the right-hand side mentions the variable $c$ twice. So any operad-theoretic approach to ring objects must involve some indirection.

And it is indeed possible to give a description of the algebraic theory of $E_\infty$-ring (or semiring) objects which is built out of small diagrams, after the same fashion as spans. We postpone this to a future paper, since some considerable extra machinery is required.

Such a description is handy, as it places the monoidal structures corresponding to addition and to multiplication on the same footing: both are defined by way of span diagrams. This means that constructions which require blurring the distinction between the two, such as the construction of the group of units $\mathrm{gl}_1$ of an $E_\infty$ ring space, become natural.

It is also possible to give a quick description of the group-completion of a theory, and the group-completion of a model of a theory.

We similarly defer discussion of these constructions to future papers.

\subsection{Comment}

Except for some of the material in Section \ref{comparison-theorems}, all the results in this paper appeared in the author's Sheffield PhD thesis \cite{PhD} under the direction of Neil Strickland. Many people are thanked in that thesis, and those sentiments are as valid as before.

\section{Preliminaries on 2-categories and quasicategories}
\label{2-cats-quasicats}

The first aim of this section is to compare various notions of 2-category, in order to match Jacob Lurie's definition of a 2-category \cite{HTT} with the classical notions.

There are several classical notions, with varying levels of strictness and laxity: as might be expected, it is simpler to construct the laxer versions, and simpler to use the stricter versions in constructions. At one end is the notion of a weak 2-category, and at the other is the notion of a strict 2-category \cite{Borceux-I}.

There is little essential difference, insofar as the work of Street and his coauthors \cites{Gordon-Power-Street, Street} (proved also in \cite{Leinster}) says that any weak 2-category can be replaced with an equivalent strict 2-category. A strict 2-category is exactly the same thing as a category enriched in categories, and we use this identification in what follows.

The second aim is to prove some basic results on quasicategories, which will be useful later on.

We use quasicategorical terminology without apology, even for arbitrary simplicial sets. Thus a 0-cell will often be called an \emph{object} of a simplicial set, and a 1-cell will often be called a \emph{morphism}.

Accordingly, when we use the word ``space'', we mean Kan complex.

\subsection{Quasicategories and $(2,1)$-categories}
\label{two-one-categories}

Let $\cC$ be a strict 2-category with all 2-cells invertible (that is, a category enriched in groupoids).

We can define its nerve in two steps. First we form a simplicial category $\bar\cC$ with $\Ob\bar\cC=\Ob\cC$, and $\bar\cC(x,y)=N\cC(x,y)$. This is a category enriched in Kan complexes, and is thus suitable for the coherent nerve construction described in \cite{HTT}, giving as our final definition that $N\cC=N^\coh(\bar\cC)$.

It is worth expanding this definition a little. We recall the definition of the simplicial categories $\fC_n$ from \cite{HTT}*{1.1.5}, and are used to define $N\cC_n=N^\coh(\bar\cC)_n=\sCat(\fC_n,\bar\cC)$. This lets us prove:
\begin{prop}
An $n$-cell in $N(\cC)_n$ consists of
\begin{itemize}
\item an $n+1$-tuple $X_0,\ldots,X_n$ of objects of $\cC$,
\item morphisms $f_{ij}:X_i\rightarrow X_j$ of $\cC$ for all $i<j$,
\item 2-cells $\theta_{ijk}:f_{jk}\circ f_{ij}\Rightarrow f_{ik}$ of $\cC$ for all $i<j<k$,
\end{itemize}
such that for any $i<j<k<l$ there is an identity on 2-cells:
\[\theta_{ijl}\circ(\theta_{jkl}*\id(f_{ij}))=\theta_{ikl}\circ(\id(f_{kl})*\theta_{ijk}):f_{kl}\circ f_{jk}\circ f_{ij}\Rightarrow f_{il}.\]
\end{prop}
\begin{proof}
As $\Ob\fC_n=\{0,\ldots,n\}$, a map of simplicial categories $\fC_n\rightarrow\bar\cC$ certainly distinguishes objects $X_0,\ldots,X_n$.

The $0$-simplices of homspaces of $\fC_n(i,j)$ correspond to subsets of the interval $\{i,\ldots,j\}$ containing both $i$ and $j$, and composition is by disjoint union. Thus they are generated under composition by the minimal subsets $\{i,j\}$. These give us the morphisms $f_{ij}:X_i\rightarrow X_j$.

The $1$-simplices of homspaces of $\fC_n(i,j)$ correspond to (the opposites of) inclusions of pairs of subsets of $\{i,\ldots,j\}$ containing both $i$ and $j$. These are generated by inclusions $\{i,k,j\}\leftarrow\{i,j\}$ under horizontal and vertical composition, providing the maps $\theta_{ijk}:f_{jk}\circ f_{ij}\Rightarrow f_{ik}$ of $\cC$.

The interchange law for horizontal and vertical composition gives us the specified identity, arising from the agreement of the composite inclusions
\begin{align*}
\{i,k,l,j\}\longleftarrow&\{i,k,j\}\longleftarrow\{i,j\},\quad\text{and}\\
\{i,k,l,j\}\longleftarrow&\{i,l,j\}\longleftarrow\{i,j\}.
\end{align*}
This identity generates all $2$-cells in $\fC_n(i,j)$, under composition.

As $\bar\cC(i,j)$ is the nerve of a groupoid, a map $\fC_n(i,j)$ is uniquely specified by its effect on the 1-skeleton, so there are no further data or identities.
\end{proof}

We refer to the identity as the \emph{compatibility condition}, and since 2-cells are invertible we can write it graphically: it says that the pasting of the following diagram is the identity 2-cell.
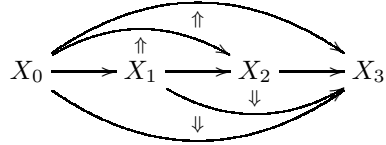
\begin{displaymath}
\xymatrix{X_0\ar[r]\ar@/^16pt/[rr]_{\Uparrow}\ar@/_26pt/[rrr]^{\Downarrow}\ar@/^26pt/[rrr]_{\Uparrow}&X_1\ar[r]\ar@/_16pt/[rr]^{\Downarrow}&X_2\ar[r]&X_3}
\end{displaymath}

We also get the following basic coherence result, which is obvious from the description above.
\begin{prop}
A lax functor $F:\cC\rightarrow\cD$ between bicategories with all 2-cells invertible yields (via passing to strict 2-categories) a map of quasicategories $N(F):N(\cC)\rightarrow N(\cD)$ between their nerves.
\end{prop}
\begin{proof}
We can replace $F$ with an equivalent functor of strict 2-categories, and then use the naturality of the nerve construction considered above.
\end{proof}

Also, this construction agrees with the construction of the nerve of a category.
\begin{prop}
Let $\cC$ be a category, regarded as a bicategory with only identities for 2-cells. Then $N(\cC)$ is the ordinary nerve of $\cC$.
\end{prop}

Now, the nerve $N(\cC)$ should be thought of as a model for $\cC$ in the world of quasicategories. Thus, we should expect it to be a (2,1)-category in the sense discussed above. This means that all all extensions of maps $\Lambda^n_k\rightarrow\cC$ to $k$-cells are unique for $n\geq 3$: its cells in degrees 3 and over are determined by those in lower degrees. The facts support our intuition:
\begin{prop}
\label{nerves-of-bicats}
The nerve $N(\cC)$ is a $(2,1)$-category.
\end{prop}
\begin{proof}
Suppose given an inner horn inclusion $\Lambda^n_k\rightarrow N(\cC)$ for $n\geq 3$, and $0<k<n$. We can recover all the 1-cells from this data: the 1-cell $X_i\rightarrow X_j$ for $i<j$ will be given by the face numbered $\alpha$ for any $\alpha\notin\{i,j,k\}$.

If $n=3$, then without loss of generality, $k=1$ (as the case $k=2$ is dealt with in a symmetric manner). We then have the following diagram:
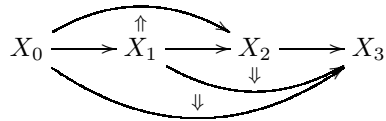
\begin{displaymath}
\xymatrix{X_0\ar[r]\ar@/^16pt/[rr]_{\Uparrow}\ar@/_26pt/[rrr]^{\Downarrow}&X_1\ar[r]\ar@/_16pt/[rr]^{\Downarrow}&X_2\ar[r]&X_3}
\end{displaymath}
This leaves us just missing the 2-cell $\theta_{023}:f_{23}\circ f_{02}\Rightarrow f_{03}$. But, since all 2-cells are invertible, we can take this to be the composite of all the 2-cells in the diagram above. In symbols, we define
\[\theta_{023}=\theta_{013}\circ(\theta_{123}*\id(f_{01}))\circ(\id(f_{23})*\theta_{012}^{-1}),\]
and this clearly fulfils the compatibility condition. This choice is clearly forced, arising as it does by solving the compatibility condition for $\theta_{023}$, and this means the extension is unique.

If $n\geq 4$, then all 2-cells are determined uniquely (indeed, $\theta_{hij}$ will be defined by face $\alpha$, for any $\alpha\notin\{h,i,j,k\}$). However, if $n=4$ there are some compatibility conditions which are not forced by the faces, and we must check that they hold.

For calculations, we omit the identity parts of our 2-cells. Then all composites are vertical composites, so we do not bother writing the $\circ$. There are five compatibility conditions coming from the faces:
\begin{align*}
\theta_{134}\theta_{123}&=\theta_{124}\theta_{234}\tag{face 0}\\
\theta_{034}\theta_{023}&=\theta_{024}\theta_{234}\tag{face 1}\\
\theta_{034}\theta_{013}&=\theta_{014}\theta_{134}\tag{face 2}\\
\theta_{024}\theta_{012}&=\theta_{014}\theta_{124}\tag{face 3}\\
\theta_{023}\theta_{012}&=\theta_{013}\theta_{123}\tag{face 4}\\
\end{align*}
Also, $\theta_{012}$ and $\theta_{234}$ commute. We can see this using the interchange law:
\begin{align*}
  \theta_{012}\theta_{234}
&=(\id(f_{24})*\theta_{012})\circ(\theta_{234}*\id(f_{02}))\\
&=(\id(f_{24})\circ\theta_{234})*(\theta_{012}\circ\id(f_{02}))
 =\theta_{234}\theta_{012}.
\end{align*}

For horn inclusions $\Lambda^4_1\rightarrow N(\cC)$, we have all coherence conditions except the one arising from face 1, and must show that from the others. But we have:
\begin{align*}
\theta_{034}\theta_{023}
&=(\theta_{014}\theta_{134}\theta_{013}^{-1})(\theta_{013}\theta_{123}\theta_{012}^{-1})\qquad\text{(faces 2 and 4)}\\
&=\theta_{014}\theta_{134}\theta_{123}\theta_{012}^{-1}\\
&=\theta_{014}(\theta_{124}\theta_{234})\theta_{012}^{-1}\qquad\text{(face 0)}\\
&=(\theta_{024}\theta_{012})\theta_{234}\theta_{012}^{-1}\qquad\text{(face 3)}\\
&=\theta_{024}\theta_{234}\qquad\text{(since $\theta_{012}$ and $\theta_{234}$ commute).}
\end{align*}

For horn inclusions $\Lambda^4_2\rightarrow N(\cC)$, we have all coherence conditions except the one from face 2. Similarly, we have:
\begin{align*}
\theta_{034}\theta_{013}
&=(\theta_{024}\theta_{234}\theta_{023}^{-1})(\theta_{023}\theta_{012}\theta_{123}^{-1})\qquad\text{(faces 1 and 4)}\\
&=\theta_{024}\theta_{012}\theta_{234}\theta_{123}^{-1}\qquad\text{(since $\theta_{012}$ and $\theta_{234}$ commute)}\\
&=(\theta_{014}\theta_{124})\theta_{234}\theta_{123}^{-1}\qquad\text{(face 3)}\\
&=\theta_{014}(\theta_{134}\theta_{123})\theta_{123}^{-1}\qquad\text{(face 0)}\\
&=\theta_{014}\theta_{134}
\end{align*}

Horn inclusions $\Lambda^4_3\rightarrow N(\cC)$ can be dealt with by an argument symmetric to that used for horn inclusions $\Lambda^4_1\rightarrow\cC$.

The fact that all structure is determined means that the extension is unique.

If $n\geq 5$, then nothing need be checked: the compatibility conditions on $X_g$, $X_h$, $X_i$ and $X_j$ will be fulfilled by face $\alpha$, for any $\alpha\notin\{g,h,i,j,k\}$.
\end{proof}

Using this nerve construction, in the sequel we shall abuse terminology systematically, and confuse a strict 2-category with its nerve, a $(2,1)$-category.

\subsection{Fibrations and extension properties of $(n,1)$-categories}
\label{n-1-categories}

In this section we prove some properties of Lurie's model for $(n,1)$-categories, from \cite{HTT}*{subsection 2.3.4}: these are those $(\infty,1)$-categories which admit all inner horn extensions $\Lambda^m_k$ uniquely where $m>n$.

It follows immediately from the definition \cite{HTT}*{2.3.4.9} that an $(n,1)$-category has at most one extension along $\partial\Delta^m\rightarrow\Delta^m$ for $m>n$; here's a strengthening of that statement:

\begin{prop}
\label{partial-delta-n-cat}
An $(n,1)$-category $\cC$ has unique liftings for $\partial\Delta^m\rightarrow\Delta^m$ where $m\geq n+2$.
\end{prop}
\begin{proof}
We can restrict the map $\partial\Delta^m\rightarrow\cC$ to a map $\Lambda^m_1\rightarrow\cC$, and lift that uniquely to a map $\Delta^m\rightarrow\cC$. This is the only candidate for a lifting; we must prove that it is compatible with the given map on all of $\partial\Delta^m$: that is, show that it agrees on the first face.

But these two $(m-1)$-cells certainly agree on the boundary of the first face (which is isomorphic to $\partial\Delta^{m-1}$) and thus agree.
\end{proof}

In a similar vein is this:
\begin{prop}
\label{outer-horn-n-cat}
An $(n,1)$-category $\cC$ has unique liftings for outer horns $\Lambda^m_0\rightarrow\Delta^m$ and $\Lambda^m_m\rightarrow\Delta^m$ where $m>n+2$.
\end{prop}
\begin{proof}
We can uniquely extend a map $\Lambda^m_0\rightarrow\cC$ to a map $\partial\Delta^m\rightarrow\cC$ using Proposition \ref{partial-delta-n-cat} on the $0$th face. Then we can uniquely extend that to a map $\Delta^m\rightarrow\cC$ using Proposition \ref{partial-delta-n-cat} again.

The case of $\Lambda^m_m$ is symmetrical.
\end{proof}

The special case of ordinary categories will be of utility later:
\begin{prop}
 \label{nerve-outer-horns}
 The nerve of a category $N\cC$ has unique liftings for outer horns $\Lambda^n_0$ and $\Lambda^n_n$ whenever $n\geq 4$.
\end{prop}

The following proposition reduces the work necessary to show that a map of $(n,1)$-categories is an acyclic Kan fibration:
\begin{prop}
\label{acyclic-kan}
A functor $\cC\rightarrow\cD$ of $(n,1)$-categories automatically has the right lifting property with respect to the maps $\partial\Delta^m\rightarrow\Delta^m$ for $m\geq n+2$.
\end{prop}
\begin{proof}
Proposition \ref{partial-delta-n-cat} gives a map $\Delta^m\rightarrow\cC$, and by \cite{HTT}*{2.3.4.9}, this is consistent with the given map $\Delta^m\rightarrow\cD$.
\end{proof}

We can say useful things about inner fibrations.
Let $F:\cC\rightarrow\cD$ be a functor between $(n,1)$-categories. We have the following simple criterion for being an inner fibration:
\begin{prop}
\label{inner-fibs}
The functor $F$ is an inner fibration if and only if it has the right lifting property for inner horns $\Lambda^m_k\rightarrow\Delta^m$ for $0<k<m\leq n$.
\end{prop}

In particular, this gives the following simple criterion for $(2,1)$-categories: $F$ is an inner fibration if, for every pair of diagrams
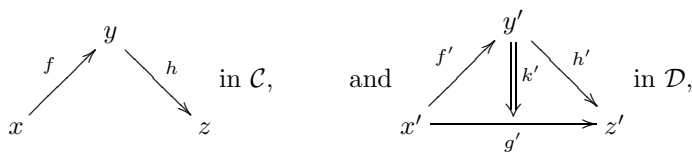
\begin{displaymath}
\vcenter{
\xymatrix{&y\ar[dr]^h&\\
          x\ar[ur]^f&&z}}
\text{in $\cC$,\hskip 1cm and}
\vcenter{
\xymatrix{&y'\ar[dr]^{h'}\Ar^{k'}[d]&\\
          x'\ar[ur]^{f'}\ar[rr]_{g'}&&z'}}
\text{in $\cD$,}
\end{displaymath}
such that $F(f)=f'$ and $F(h)=h'$, there is a 1-cell $g:x\rightarrow z$ and 2-cell $k:h\circ f\Rightarrow g$ such that $F(g)=g'$ and $F(k)=k'$.

We now switch our attention to the more intricate notion of a cartesian fibration. These are analogues of the classical notion of a Grothendieck fibration of categories. They are morphisms of simplicial sets which describe a family of quasicategories varying in a contravariant functorial manner over a base quasicategory. Following Lurie \cite{HTT}, we make the following definition:
\begin{defn}
\label{defn-cartesian-fibration}
A \emph{cartesian fibration} $p:\cC\rightarrow\cD$ of quasicategories is a functor which is both an inner fibration and is such that, for every 1-morphism $f:x\rightarrow y$ (meaning a 1-cell $f\in\cD_1$ with $d_0f=x$ and $d_1f=y$) and every lift $\tilde y$ of $y$ to $\cC$ (meaning an 0-cell $\tilde y\in\cC_0$ with $p(\tilde y)=y$), there is a $p$-cartesian morphism $\tilde f$ in $\cC$ which maps to $f$ under $p$.

In turn, a \emph{$p$-cartesian morphism} $f:a\rightarrow b\in\cC_1$ is one such that the natural map
\[L_f:\cC_{/f}\longrightarrow \cC_{/y}\timeso{\cD_{/y}}\cD_{/f},\]
where $y=f(b)$, is an acyclic Kan fibration.

There is a dual notion of a \emph{cocartesian morphism} and a \emph{cocartesian fibration}: a cocartesian fibration describes a family of quasicategories varying covariantly functorially over a base quasicategory. Given $p:\cC\rightarrow\cD$, a cocartesian morphism in $\cC$ is a cartesian morphism for $p^\op:\cC^\op\rightarrow\cD^\op$, and $p$ is a cocartesian fibration if $p^\op$ is a cartesian fibration.
\end{defn}

Lurie proves that overcategories of $(n,1)$-categories are $(n,1)$-categories \cite{HTT}*{Lemma 1.2.17.10}. The class of $(n,1)$-categories is not closed under fibre products. But the following lemma does most of the work for us:

\begin{prop}
The class of simplicial sets which are the coskeleton of their $k$-skeleton is closed under all limits.
\end{prop}
\begin{proof}
The $n$-skeleton functor $\skel_n$ visibly preserves limits, and the $n$-coskeleton functor preserves limits since it is right adjoint to $\skel_n$. Given this, this category is closed under limits.
\end{proof}

Indeed, more is true:
\begin{prop}
For any $0\leq i\leq n$, the class of simplicial sets with unique liftings for maps $\Lambda^n_i\rightarrow\Delta^n$ is closed under limits.
\end{prop}
\begin{proof}
The class of such simplicial sets is closed under products since a lifting for the product is just the product of liftings of the factors; we'll verify it for fibre products too.

So if $X$, $Y$ and $Z$ are simplicial sets with unique extensions for the map $\Lambda^n_i\rightarrow\Delta^n$, then I claim that $X\timeso{Z}Y$ has unique liftings for it too. Indeed, a map $\Lambda^n_i\rightarrow X\timeso{Z}Y$ consists of maps $\Lambda^n_i\rightarrow X,Y$ whose composites with the maps from $X$ and $Y$ to $Z$ agree.

These extend uniquely to maps $\Delta^n\rightarrow X,Y$. Their composites with the maps to $Z$ are both extensions of our map $\Lambda^n_i\rightarrow Z$. But such extensions are unique, and so they agree. Thus these maps assemble to a unique extension $\Delta^n\rightarrow X\timeso{Z}Y$.
\end{proof}

These combine to prove the following:
\begin{prop}
Let $p:\cC\rightarrow\cD$ be a map between $(n,1)$-categories. A morphism $f:x\rightarrow y$ in $\cC_1$ is $p$-cartesian if and only if the morphism
\[\cC_{/f}\longrightarrow\cC_{/y}\timeso{\cD_{/py}}\cD_{/pf}\]
has the right lifting property for all maps $\partial\Delta^m\rightarrow\Delta m$ for $m\leq n+1$.
\end{prop}
\begin{proof}
By the results above, both sides are $(n,1)$-categories; we thus apply Proposition \ref{acyclic-kan} to show the higher lifting conditions are automatic.
\end{proof}

\subsection{Overcategories and limits of simplicial sets}

In this section we study the relationship between Joyal's over construction (described in \cite{HTT}*{Lemma 1.2.9.2}), and limits of simplicial sets. The result is that taking overcategories commutes with taking limits of simplicial sets, in the following sense:
\begin{prop}
\label{limits-and-overs}
  Suppose we have an (ordinary) finite category $D$, to be thought of as a diagram category, and a diagram $F:D\rightarrow\sSet$. Suppose also that we have a cone on it: a simplicial set $K$ and a natural transformation $\theta:K\Rightarrow F$ to $F$ from the constant functor at $K$.

  We then get a map $\bar\theta:K\rightarrow\limit F$ from $K$ to the limit of the diagram $F$.

  We then have that the over construction commutes with limits in the sense that
  \[(\limit F)_{/\bar\theta}\isom{\limit}_x(F(x)_{/\theta_x}).\]
\end{prop}
\begin{proof}
We consider maps from a fixed simplicial set $Y$; it's then just a straightforward check:
\begin{align*}
       \sSet(Y,(\limit F)_{/\bar\theta}) 
\isom &\sSet_\theta(Y\star K,\limit F)\\
\isom &{\limit}_{x\in D}\sSet_{\theta_x}(Y\star K,F(x))\\
\isom &{\limit}_{x\in D}\sSet(Y,F(x)_{/\theta_x})\\
\isom &\sSet(Y,{\limit}_{x\in D}(F(x)_{/\theta_x}).
\end{align*}
\begin{comment}
We prove it for binary product diagrams and equaliser diagrams; it follows that it holds for all finite limit diagrams. These are both straightforward checks.

Firstly, suppose we have a binary product diagram: the given data can be deassembled and written as two simplicial sets $C$ and $D$, and maps $f:K\rightarrow C$ and $g:K\rightarrow D$. In these terms, we are trying to show that
\[(C\times D)_{/(f\times g)}\isom C_{/f}\times D_{/g}.\]
But for any $Y$, we have
\begin{align*}
\Hom(Y,(C\times D)_{/(f\times g)}) &= \Hom_{f\times g}(Y\star K,C\times D)\\
                                   &= \Hom_f(Y\star K,C)\times\Hom_g(Y\star K,D)\\
                                   &= \Hom(Y,C_{/f})\times\Hom(Y,D_{/g})\\
                                   &= \Hom(Y,C_{/f}\times D_{/g}).
\end{align*}

The argument for equalisers is very similar. The given data can be written as follows:
\begin{displaymath}
  \xymatrix{&K\ar[dl]_f\ar[d]^g\ar[dr]^h&\\
            E\ar[r]_e&C\ar@<2pt>[r]^a\ar@<-2pt>[r]_b&D,}
\end{displaymath}
where the bottom row is an equaliser diagram. We seek to show that $E_f$ is the equaliser of $C_g\rightrightarrows D_h$.

Again, for any $Y$, we have $\Hom(Y,E_{/f})=\Hom_f(Y\star K,E)$. However, the right-hand-side is the equaliser of the two maps
\[\xymatrix{\Hom_g(Y\star K,C)\ar@<2pt>[r]\ar@<-2pt>[r]&\Hom_h(Y\star K,D),}\]
and the claim follows as before.
\end{comment}
\end{proof}

We continue this analysis to derive a corresponding result for finite products and cartesian morphisms:
\begin{prop}
\label{prods-and-carts}
 If $q_1:\cC_1\rightarrow\cD_1$ and $q_2:\cC_2\rightarrow\cD_2$ are maps of quasicategories, then, defining $q=q_1\times q_2:\cC_1\times\cC_2\rightarrow\cD_1\times\cD_2$, the $q$-cartesian morphisms (as defined in subsection \ref{defn-cartesian-fibration}) are exactly the products of $q_1$-cartesian morphisms and $q_2$-cartesian morphisms.
\end{prop}
\begin{proof}
 We must relate a pullback of overcategories of products to a product of pullbacks of overcategories. The pullbacks and products commute, as usual; Proposition \ref{limits-and-overs} provides that the formation of products and of overcategories commute.
\end{proof}

Lastly, we provide a handy alternative description of overcategories. Given a quasicategory $\cC$, consider the map $p:\cCD\rightarrow\cC$ induced by evaluation of the terminal vertex of $\Delta^1$.

\begin{prop}
\label{alternative-overcat}
The fibres $p^{-1}(x)$ of $p$ are equivalent to $\cC_{/x}$.
\end{prop}
\begin{proof}
Maps $K\rightarrow p^{-1}(x)$ are maps $K\times\Delta^1\rightarrow\cC$ sending $K\times 1$ to $x$, or equivalently are maps $(K\times\Delta)/(K\times 1)\rightarrow\cC$ pointed at $x$. But there is a map $(K\times\Delta^1)/(K\times 1)\rightarrow K\star 1$, so there is a map $\cC_{/x}\rightarrow p^{-1}(x)$.

Moreover, since this map $(K\times\Delta^1)/(K\times 1)\rightarrow K\star 1$ is a strong deformation retract. We immediately get equivalences of homspaces. The required result then follows from the equivalence of simplicial categories and quasicategories. 
\end{proof}

\subsection{Limits in undercategories}

In this section, we show how quasicategorical limits in undercategories are related to limits in the original quasicategory.

\begin{prop}
\label{forgetting-from-undercat-preserves-limits}
Let $\cC$ be any quasicategory with limits, and let $f:D\rightarrow\cC$ be any diagram in it. The ``forgetful'' map $\cC_{D/}\rightarrow\cC$ preserves limits.
\end{prop}
\begin{proof}
Suppose we have a diagram of shape $K$ in $\cC_{D/}$. Postcomposition with the forgetful map gives a diagram $K\rightarrow\cC$, which admits a limit $1\star K\rightarrow\cC$; and the map $\cC_{/(1\star K)}\rightarrow\cC_{/K}$ is acyclic Kan.

Our diagram $K\rightarrow\cC_{D/}$ is equivalent to a diagram $D\rightarrow\cC_{/K}$. By the acyclic Kan condition, this gives us a map $D\rightarrow\cC_{/(1\star K)}$, or equivalently, $1\star K\rightarrow\cC_{D/}$.

We must merely show that this is indeed a limit: that $\cC_{D//(1\star K)}\rightarrow\cC_{D//K}$ is acyclic Kan. However, given $I\rightarrow J$ a cofibration, there is a bijective correspondence between squares of the two following sorts:
\begin{displaymath}
 \xymatrix{I\ar[d]\ar[r]&J\ar[d]\\
           \cC_{D//(1\star K)}\ar[r]&\cC_{D//K}}
\qquad\text{and}\qquad
 \xymatrix{D\star I\ar[d]\ar[r]&D\star J\ar[d]\\
           \cC_{/(1\star K)}\ar[r]&\cC_{/K}}.
\end{displaymath}
Since $\cC_{/(1\star K)}\rightarrow\cC_{/K}$ is acyclic Kan, we have a diagonal filler on the right, which gives us one on the left, too.
\end{proof}

By a straightforward dualisation, of course we also get:
\begin{prop}
\label{forgetting-from-overcat-preserves-colimits}
Let $\cC$ be any quasicategory and $f:D\rightarrow\cC$ any diagram in it. Then the forgetful map $\cC_{/D}\rightarrow\cC$ preserves colimits.
\end{prop}

Lastly, we have the following useful result:
\begin{prop}
 \label{interchange-of-limits}
 If $\cC$ is a complete quasicategory, and $f:X\times Y\rightarrow \cC$ is a diagram in $\cC$, then we have the usual interchange-of-limits isomorphisms
 \[\lim_X\lim_Y\isom\lim_{X\times Y}\isom\lim_Y\lim_X.\]
\end{prop}
\begin{proof}
 This follows from \cite{HTT}*{Prop 4.2.2.7}.
\end{proof}

\subsection{Limits and colimits in $\Spaces$}

This subsection serves two purposes. Firstly, it gives a straightforward construction of homotopy pullbacks in $\Spaces$. Then it gives a couple of basic properties of colimits in the quasicategory of spaces; they are both recognisable results in the discrete case, where they reduce to results about ordinary colimits in the ordinary category of sets.

\begin{defn}
\label{homotopy-pullbacks-spaces}
We use the following a natural model for the quasicategorical pullback: we write $E3$ for the standard contractible simplicial set on three points $l$, $m$ and $r$ (with one $n$-simplex for each $(n+1)$-tuple of vertices). Then we define
\[\cC_1\times^h_\cE\cC_2 = (\cC_1\times\cC_2)\times_{(\cE\times\cE)}\Map(E3,\cE).\]
Here the morphism $\Map(E3,\cE)\rightarrow \cE\times\cE$ is given by evaluation on $l$ and $r$.
\end{defn}

We also write down the structure maps of the limiting cone:
\begin{displaymath}
\xymatrix{\cC_1\times^h_\cE\cC_2\ar[d]_{p_2}\ar[dr]^f\ar[r]^{p_1}&\cC_1\ar[d]\\
          \cC_2\ar[r]&\cE.}
\end{displaymath}
The maps $p_1$ and $p_2$ are the evident projections, and $f$ is induced by the map $\Map(E3,\cE)\rightarrow\cE$ given by evaluation at $m$. The homotopies between the composites $\cC_1\times^h_\cE\cC_2\rightarrow\cE$ are induced by the equivalences $l\isom m$ and $m\isom r$ in $E3$. This is of course merely a simplicial variant of the standard topological construction of the homotopy pullback \cite{Hirschhorn}*{18.1.7}.

Now we turn to colimits. We start with an easy observation:
\begin{prop}
\label{colims-prods-commute}
Colimits in $\Spaces$ commute with products.
\end{prop}
\begin{proof}
We invoke \cite{HTT}*{Corollary 4.2.4.8} to show that it suffices to do this in the simplicial category $\Spaces^\Delta$, for arbitrary coproducts and homotopy pushouts.

Both of these are easy checks. Indeed,
\[\coprod_{a\in A}(Z\times X_a)\isom Z\times\coprod_{a\in A}X_a.\]
Also, since the formation of mapping cylinders commutes with products, homotopy coequalisers do also.
\end{proof}

This allows us to prove:
\begin{prop}
\label{products-and-colimits}
Let $F:K\rightarrow\Spaces$ and $G:L\rightarrow\Spaces$ be diagrams in the quasicategory of spaces. Then the diagram
\[F\times G:K\times L\longrightarrow\Spaces\times\Spaces\stackrel{\mathrm{prod}}{\longrightarrow}\Spaces\]
has colimit given by $\colim(F\times G)=\colim(F)\times\colim(G)$.
\end{prop}
\begin{proof}
This is an easy calculation, using \ref{colims-prods-commute} twice:
\begin{align*}
\colim(F\times G) &= \colim_{k\in K}\colim_{l\in L}\left(F(k)\times G(l)\right) \\
                  &= \colim_{k\in K}\left(F(k)\times\colim_{l\in L}G(l)\right) \\
                  &= \colim_{k\in K}\left(F(k)\times\colim(G)\right) \\
                  &= \left(\colim_{k\in K}F(k)\right)\times\colim{G} \\
                  &= \colim(F)\times\colim(G).
\end{align*}
\end{proof}

\subsection{Mapping cylinders in quasicategories}
\label{mapping-cylinders-quasicategories}

As discussed above, cocartesian fibrations over $\cC$ are equivalent to functors $\cC\rightarrow\Cinfty$. In the case where $\cC=\Delta^1$, we shall later have need of a direct way of replacing functors between quasicategories $F:\cA\rightarrow\cB$ with cocartesian fibrations $\cE\rightarrow\Delta^1$. This approach is, in fact, a special case of Lurie's \emph{relative nerve} construction \cite{HTT}*{3.2.5}; however it may nevertheless be helpful to have a self-contained account of it.

Firstly, we note that maps $\Delta^n\rightarrow\Delta^1$ are classified by the preimages of the two vertices of $\Delta^1$. Thus we write $\Delta^{I\sqcup J}$ for a simplex with the implied map to $\Delta^1$ sending $I$ to $0$ and $J$ to $1$.

We define our model $p:\cE\rightarrow\Delta^1$ by giving that
\begin{displaymath}
\left\{\text{maps}\quad\vcenter{\xymatrix{\Delta^{I\sqcup J}\ar[r]\ar[dr]&\cE\ar[d]\\&\Delta^1}}\right\}
= \left\{\text{diagrams}\quad\vcenter{\xymatrix{\Delta^I\ar[d]\ar[r]^i&\Delta^{I\sqcup J}\ar[d]\\\cA\ar[r]_F&\cB}}\right\},
\end{displaymath}
where the map $i$ is that induced by the evident inclusion $I\rightarrow I\sqcup J$.

We now show by parts that this serves for us. To start with, it is straightforward to check that the preimages of the vertices are isomorphic to $\cA$ and $\cB$ respectively.

\begin{prop}
\label{mapping-cylinders-1}
The map $p:\cE\rightarrow\Delta^1$ is an inner fibration.
\end{prop}
\begin{proof}
We need to show that an inner horn $\Lambda^{I\sqcup J}_k\rightarrow\cE$ extends to a full simplex $\Delta^{I\sqcup J}\rightarrow\cE$. Since the preimages of both vertices are quasicategories,  we need only concern ourselves with the case where $I$ and $J$ are both nonempty.

In either case, the faces we have include a full map $\Delta^I\rightarrow\cA$; this merely leaves us with an inner horn extension $\Delta^{I\sqcup J}\rightarrow\cB$, which is possible as since $\cB$ is a quasicategory.
\end{proof}

\begin{prop}
\label{mapping-cylinders-2}
For any element $a\in\cA_0$, there is a $p$-cocartesian morphism of $\cE_1$ whose 0th vertex is $a$, and which lies over the nontrivial 1-cell of $\Delta^1$.
\end{prop}
\begin{proof}
We chose the 1-cell $\alpha$ given by
\begin{displaymath}
\xymatrix{\Delta^0\ar[r]^0\ar[d]_a&\Delta^1\ar[d]^a\\
 \cA\ar[r]_F&\cB.}
\end{displaymath}

Now we go on to show that this is indeed $p$-cocartesian: that the morphism
\[\cE_{\alpha/}\longrightarrow \cE_{a/}\timeso{\cE}\cB\]
is acyclic Kan.

A diagram 
\begin{displaymath}
\xymatrix{\partial\Delta^n\ar[d]\ar[r]&\Delta^n\ar[d]\\
\cE_{\alpha/}\ar[r]&\cE_{a/}\timeso{\cE}\cB}
\end{displaymath}
unravels to give a diagram
\begin{displaymath}
\xymatrix{&1\star(\emptyset\star\partial\Delta^n)\ar[dl]\ar[dr]&\\
1\star(1\star\partial\Delta^n)\ar[dr]&&1\star(\emptyset\star\Delta^n)\ar[dl]\\
&\cE.&}
\end{displaymath}
Since these agree on the part mapping to $T$, the problem is to extend a map from $(\Delta^1\star\partial\Delta^n)\cup(1\star\Delta^n)\isom\Lambda^{2+n}_0$ to a map from $\Delta^{2+n}$. However, the edge between the first two vertices is a degeneracy, so the given map factors through $\Delta^{1+n}$ and can be extended to $\Delta^{2+n}$ via the map $\Delta^{2+n}\rightarrow\Delta^{1+n}$ which collapses the first two vertices.
\end{proof}

We can now prove:
\begin{prop}
\label{mapping-cylinders}
The map $p:\cE\rightarrow\Delta^1$ is an cocartesian fibration.
\end{prop}
\begin{proof}
We have just shown in Proposition \ref{mapping-cylinders-1} that $p$ is an inner fibration; there remains the question of cocartesian lifts. We only need to provide cocartesian lifts over nonidentity cells of $\Delta^1$; and Proposition \ref{mapping-cylinders-2} does this.
\end{proof}

\section{Generalities on algebraic theories}
\label{theories-generalities}

In this section we generalise the notion of an algebraic theory, to the setting of quasicategories. We base our account mostly on that given by Borceux \cite{Borceux-II} for the classical case.

We will have need of the quasicategory $\Spaces$. By this we mean the quasicategory obtained as the coherent nerve of the simplicial category of Kan complexes (together with their mapping complexes), as is used in \cite{HTT}*{1.2.16.1}. However, there are other natural constructions of equivalent quasicategories, just as there are several natural model categories Quillen equivalent to the standard model structure on topological spaces. Of course it will not matter which is used.

We regard $\Set$ as being the full subquasicategory on the discrete spaces; this is evidently equivalent to the standard notion. We shall use the adjective \emph{discrete} frequently to describe phenomena which occur over $\Set$ rather than the whole of $\Spaces$.

We should also say, once and for all, what we mean by this:
\begin{defn}
\label{full-subquasicategory}
A \emph{full subquasicategory} of a quasicategory is a maximal subquasicategory with its set of $0$-cells. We shall also call this \emph{$1$-full}; an $n$-full subquasicategory is a maximal subquasicategory with that particular set of $k$-cells for all $k<n$.
\end{defn}

\subsection{Theories and models}

Given a quasicategory $\cC$ and a subset $S\subset\cC_0$, we say that \emph{$\cC$ is generated by finite products from $S$} if every object of $\cC$ is equivalent to a finite product of elements in $S$.

If $\cC$ is generated by finite products from $\{g\}$, we refer to $g$ as a \emph{finite product generator} of $\cC$.

\begin{defn}
\label{theory-defn}
An \emph{algebraic theory} is a quasicategory $T$ which admits finite products and which has a finite product generator $g$.
\end{defn}

\begin{defn}
A morphism of algebraic theories $T\rightarrow S$ is a product-preserving functor $T\rightarrow S$ which sends generators to generators.
\end{defn}
By product-preserving, I mean ``taking finite product diagrams to finite product diagrams''; I suppose this is the standard meaning, but I am departing from tradition in making this plain.

It follows immediately from the definition that a morphism of theories $T\rightarrow S$ is essentially surjective.

We can define a quasicategory $\Theories$ of algebraic theories to be the $2$-full subquasicategory of $\Cinfty$ on the theories and morphisms of theories.

Theories are not much use without introducing a notion of model:
\begin{defn}
A \emph{model} (in spaces) of a theory $T$ is a product-preserving functor $\phi:T\rightarrow\Spaces$. The quasicategory of models of $\cC$ is the full subquasicategory $\Mod(T)$ of $\Map(T,\Spaces)$ on the product-preserving objects.

In exactly the same way, if $\cU$ is any category with all finite products, we define the quasicategory of \emph{models of $T$ in $\cU$} to be the quasicategory $\Mod(T,\cU)$ of product-preserving functors $T\rightarrow\cU$ .
\end{defn}

By abuse of notation, we often write $1$ for a choice of generator of a theory. Given a model $\phi$, we sometimes say that $\phi(1)$ is the \emph{underlying object} of the model, and that giving such a functor $\phi$ is equipping the space $\phi(1)$ with a \emph{$T$-structure}.

We now observe that this does indeed generalise Lawvere's original notion (which is discussed in \cite{Lawvere}). To do this, we introduce terminology for this special case:
\begin{defn}
An algebraic theory $T$ is \emph{discrete} if $T$ is, in fact, an ordinary category.

If a theory $T$ is discrete, we say that a model $M$ of $T$ is \emph{discrete} if it is valued in sets (regarded as a subcategory of spaces).
\end{defn}

The rationale is that, since the subquasicategory of simplicial sets on the discrete objects is equivalent to the ordinary category of sets, if $T$ is an ordinary category, then a functor (of quasicategories) $T\rightarrow\Set\subset\Spaces$ is just an ordinary functor.

So an algebraic theory in the sense of Lawvere, which is an ordinary category $T$ generated by finite products of a single object, is the same thing as a discrete algebraic theory in the sense defined here. Moreover, a model of an algebraic theory in the sense of Lawvere is the same thing as a discrete model of the corresponding discrete algebraic theory.

A morphism $f:T\rightarrow U$ of theories induces a functor $f^*:\Mod(T)\rightarrow\Mod(U)$ by precomposition.

% \subsection{Models in quasicategories}
\label{models-in-quasicategories}

Once could study models of a theory $T$ in quasicategories simply by using $\Mod(T,\Cinfty)$ as defined above. Usually, we will employ an equivalent but more easily manipulated definition:
\begin{defn}
A \emph{model of $T$ in quasicategories} is a cocartesian fibration over $T$ that is classified by a product-preserving functor. We call such cocartesian fibrations \emph{productive}.
\end{defn}

We also need to deal with maps between models:
\begin{defn}
We define the quasicategory $\Mod^\fib(T)$ of models in quasicategories of a theory $T$. This is the subquasicategory of the overcategory $(\Cinfty)_{/T}$, consisting of all those cells whose vertices are productive cocartesian fibrations over $T$, and whose edges are product-preserving functors taking cartesian morphisms to cartesian morphisms.
\end{defn}

A morphism $f:T\rightarrow U$ of theories induces a functor $f^*:\Mod^\fib(U)\rightarrow\Mod^\fib(T)$ induced by pulling back the cocartesian fibrations.  By the results of \cite{HTT}*{Section 2.4.2 and Chapter 3}, there is an equivalence between $\Mod^\fib(T)$ and the previously defined notion $\Mod(T,\Cinfty)$.

\subsection{Multisorted theories}

Occasionally, one has need to consider axioms for algebraic structures with several underlying objects, and maps between them.

Accordingly, we define:
\begin{defn}
 Let $X$ be a set. A \emph{multisorted theory with sorts indexed by $X$} consists of a quasicategory $T$ and a map $X\rightarrow T_0$, such that $T$ is generated by finite products from the image of $X$.

 We refer to multisorted theories with sorts indexed by $\{1,\ldots,n\}$ as being \emph{$n$-sorted theories}.

 A \emph{model} in $\cC$ of a multisorted theory $T$ with sorts indexed by $X$ is a product-preserving functor from $T$ to $\cC$.
\end{defn}

By way of trivial example, if $T_1,\ldots,T_n$ are theories with generators $g_1,\ldots,g_n$, then the product $T_1\times\cdots\times T_n$ is the $n$-sorted theory (with generators consisting of the elements of the form $(1,\ldots,1,g_i,1,\ldots,1)$) whose models are tuples consisting of a model of each of the theories $\{T_i\}$:
\[\Mod(T)=\Mod(T_1)\times\cdots\times\Mod(T_n).\]

In the main, the basic results for single-sorted theories carry over to $n$-sorted theories as one would expect, and we shall not write them out.

We can regard all the categories of multisorted theories as forming subcategories of the quasicategory $\Cinftypp$ of quasicategories with all finite products, product-preserving functors, and homotopies between them.  In particular, $\Mod(T;\cU)=\Cinftypp(T,\cU)$.

\begin{prop}
\label{models-has-an-adjoint}
Fix a quasicategory $\cU$ with finite products. The quasifunctor $\Funpp(-,\cU):(\Cinftypp)^\op\rightarrow\Cinfty$, which assigns to each theory its category of models, has a left adjoint.
\end{prop}
\begin{proof}
Our proof proceeds by exhibiting an adjunction in detail. However, I consider that this motivating argument is considerably more enlightening. The idea is that
\[\Funpp(T,\Fun(\cC,\cU))\isom\Fun(\cC,\Funpp(T,\cU))\]
since products are computed pointwise. This means that
\[\Theories^\op(\Fun(\cC,\cU),T)\isom\Cinfty(\cC,\Mod(T;\cU)).\]
which is exactly the equivalence on homspaces required for an adjunction.

% The left adjoint is the functor $\Fun(-,\cU)^\op:\Cinfty\rightarrow(\Cinftypp)^\op$; we build an adjunction by hand. It seems the methods employed here could be useful in rather greater generality.

By \cite{HTT}*{Section 5.2}, an adjunction is represented by a cartesian and cocartesian fibration over $\Delta^1$.

Now the maps $\Delta^n\rightarrow\Delta^1$ are described by the preimages of the vertices: they are equivalent to decompositions $\Delta^n=\Delta^i\star\Delta^j$ where $i,j\geq-1$ and $i+j=n$. So we can define a simplicial set $\cD$ over $\Delta^1$ by giving a compatible set of homsets $\sSet_{\Delta^1}(X\star Y,\cD)$.

We define $\cD$ by letting $\sSet_{\Delta^1}(X\star Y,\cD)$ consist of maps $c:X\rightarrow\Cinfty$ and $a:Y^\op\rightarrow\Cinftypp$, together with a map $f:(X\times Y^\op)\star 1\rightarrow\Cinfty$, which are equipped with a natural equivalence with $(c\times a):X\times Y^\op\rightarrow\Cinfty$ when restricted to $X\times Y^\op$, and which send the extra point $1$ to $U$.

Writing $\pi$ for the projection $\cD\rightarrow\Delta^1$, we easily see that $\pi^{-1}(0)=\Cinfty$ and $\pi^{-1}(1)=\Cinftypp$. We must show that $\pi$ is a bicartesian fibration, to show that it represents an adjunction.

We split this into two parts. Firstly we show that $\pi$ has the inner Kan lifting property:
\begin{claim}
The morphism $\pi$ is an inner fibration.
\end{claim}
\begin{innerproof}{of claim}
For greater flexibility, we index our simplices by finite linearly ordered sets in this argument.

So, given finite linearly ordered sets $I$ and $J$, and $k$ some internal element of the concatenation $I\sqcup J$ , we must provide a lifting
\begin{displaymath}
\xymatrix{\Lambda^{I\sqcup J}_k\ar[d]\ar[r]&\Delta^{I\sqcup J}\ar[d]\dar[dl]\\
  \cD\ar[r]&\Delta^1.}
\end{displaymath}

If either $I$ or $J$ have no elements, this clearly reduces to the statement that the preimages $\Cinfty$ and $(\Cinftypp)^\op$ of the endpoints of $\Delta^1$ are both quasicategories. 

Supposing otherwise, we assume without loss of generality that $k\in I$ (the case $k\in J$ is symmetrical). Observing that
\[\Lambda^{I\sqcup J}_k = (\Lambda^I_k\star\Delta^J)\cup_{(\Lambda^I_k\star\partial\Delta^J)}(\Delta^I\star\partial\Delta^J),\]
we get that a morphism $f:\Lambda^{I\sqcup J}_k\rightarrow\cD$ consists of maps $c:\Delta^I_k\rightarrow\Cinfty$, $a:(\Delta^J)^\op\rightarrow\Cinftypp$, and a map
\[\left((\Lambda^I_k\times\Delta^J)\cup_{(\Lambda^I\times\partial\Delta^J)}(\Delta^I\times\partial\Delta^J)\right)\star 1\longrightarrow\Cinfty.\]
Using \cite{JoyalTierneyBook}*{3.2.2}, we see that the inclusion $(\Lambda^I_k\times\Delta^J)\cup(\Delta^I\times\partial\Delta^J)\rightarrow\Delta^I\times\Delta^J$ is anodyne: it's a composite of horn extensions. Tracing the argument carefully (using that $k$ is not the initial object of $I$) we see that no horn extensions of shape $\Lambda^r_0\rightarrow\Delta^r$ are required, even in the case where $k$ is terminal in $I$. Since we are doing that extension working over $U$, only inner horn extensions are needed.
\begin{comment}
To warm up, for $\Lambda^{1+1+0}_1$, the first face gives us $c:C_0\rightarrow C_1$, and the second face gives us $a=A_0$ and $f_{1,0}:C_1\times A_0\rightarrow U$. We can use an extension of shape $\Lambda^2_1\rightarrow\Delta^2$ to provide a full $2$-cell in $\cD$:
\begin{displaymath}
\xymatrix{C_0\times A_0\ar[rr]^{c\times A_0}\dar[dr]&&C_1\times A_0\ar[dl]^{f_{1,0}}\\
  &U&}
\end{displaymath}
The case of $\Lambda^{0+1+1}_1$ is of course symmetrical to this.
\end{comment}
\end{innerproof}

And now secondly we show the existence of cartesian and cocartesian lifts. Since there is only one nontrivial 1-cell $01\in\Delta^1_1$, we must merely show:
\begin{claim}
For any object $A\in(\Cinftypp)_0$, there is a cartesian morphism of $\cD$ over $01$ with target $A$; for any object $C\in(\Cinfty)_0$, there is a cocartesian morphism of $\cD$ over $01$ with source $C$.
\end{claim}
\begin{innerproof}{of claim}
Given $C\in(\Cinfty)_0$, we must give a cocartesian 1-cell in $\cD$ from it which lies over the nondegenerate 1-cell of $\Delta^1$; we take the cell consisting of $c=C\in(\Cinfty)_0$, $a=\Fun(C,\cU)\in(\Cinftypp)_0$, and $f\in(\Cinfty)_1$ representing the evaluation map $C\times\Fun(C,\cU)\rightarrow \cU$.

Similarly, given $A\in(\Cinftypp)_0$, the cartesian 1-cell in $\cD$ consists of $c=\Funpp(A,\cU)\in(\Cinfty)_0$, $a=A\in(\Cinftypp)_0$, and $f\in(\Cinfty)_1$ representing the evaluation map $\Funpp(A,\cU)\times A\rightarrow \cU$.

The proofs that these are indeed cocartesian and cartesian respectively are very similar. We aim to show that the morphism
\[\cD_{(c,a,f)/}\longrightarrow\cD_{c/}\timeso{\cD}\Cinftypp\]
is acyclic Kan. This unravels to the requirement that we can extend two compatible maps $\Delta^n\rightarrow{\Cinfty}_{/\cU}$ and $\partial\Delta^n\rightarrow{\Cinfty}_{/(\cC\times\Fun(\cC,\cU)\rightarrow \cU)}$ to a map $\Delta^n\rightarrow{\Cinfty}_{/(\cC\times\Fun(\cC,\cU)\rightarrow \cU)}$, with a requirement that all the maps we supply are product-preserving.

That we can do so follows immediately from the adjunction (in the quasicategorical sense) of the functors $(\cC\times-)$ and $\Fun(\cC,-)$ for $n\geq 1$, and is a quick check in the case $n=0$.
\end{innerproof}
This completes the proof.
\end{proof}

As an immediate corollary, we get
\begin{prop}
\label{models-colimits-to-limits}
The ``models'' functor $\Mod(-,\cU)$ takes colimits of theories to limits of their quasicategories of models.
\end{prop}
We shall show in Proposition \ref{theories-cocomplete} that colimits of theories exist; and thus this will be a helpful tool.

\subsection{The initial theory}

\begin{prop}
\label{finop-is-trivial}
Let $\cU$ be any quasicategory with finite products. Then
\[\Funpp(\Finop,\cU) \isom \cU;\]
in other words, a product-preserving functor from $\Finop$ merely picks out an object of $\cU$.
\end{prop}
\begin{proof}
The inclusion $*\rightarrow\Finop$ induces a functor $I:\Funpp(\Finop,\cU)\rightarrow\cU$. We claim that this is an equivalence.

Moreover, we claim that an inverse equivalence is given by right Kan extension. We must first check that right Kan extension does indeed define a functor $\cU\rightarrow\Funpp(\Finop,\cU)$.

The right Kan extension $F$ of $u:*\rightarrow\cU$ is given by
\[F(X)=\lim\left(\left((\Finop)_{X/}\timeso{\Finop}*\right)\rightarrow *\stackrel{u}{\rightarrow}\cU\right).\]
Now the category $(\Finop)_{X/}\timeso{\Finop}*$ is the discrete category on the set of maps $*\rightarrow X$, so this limit is the $X$-fold product of $u$. Hence this Kan extension is indeed product-preserving.

Morever, by this analysis, any extension of $*\rightarrow\cU$ to a product-preserving functor $\Finop\rightarrow\cU$ is a right Kan extension. But there is a contractible space of such extensions. Hence the functor $I$ has all its preimages contractible, and is thus an equivalence.
\end{proof}

\subsection{Properties of quasicategories of models}

These categories of models have good properties:
\begin{prop}
 \label{models-limits}
 If $\cC$ is a theory, then the quasicategory $\Mod(T)$ is complete, with limits computed pointwise.
\end{prop}
\begin{proof}
 The quasicategory $\Fun(T,\Spaces)$ is complete, with limits computed pointwise. By Proposition \ref{interchange-of-limits} showing that limits can be interchanged, the limit of a diagram from $\Mod(T)$ is again in $\Mod(T)$, and is thus the limit in $\Mod(T)$.
\end{proof}

\begin{prop}
 \label{models-functor-limits}
 Given a morphism of theories $f:T\rightarrow U$, the pullback functor $f^*:\Mod(U)\rightarrow\Mod(T)$ preserves limits.
\end{prop}
\begin{proof}
 The pullback functor $\Map(U,\Spaces)\rightarrow\Map(T,\Spaces)$ evidently preserves limits, since they're defined pointwise. The result follows, since limits in $\Mod(T)$ are just limits in $\Map(T,\Spaces)$ (and the same for $U$), and this pullback functor restricts to our desired one.
\end{proof}

We recall from \cite{HTT}*{Section 5.3.1} the notion of a filtered simplicial set. This is equivalent for having liftings for all maps $A\rightarrow A\star 1$, where $A$ is the nerve of a finite poset. A filtered colimit is then just a colimit on a filtered diagram.

\begin{prop}
 The category $\Mod(T)$ has filtered colimits, which are computed pointwise.
\end{prop}
\begin{proof}
 This is the same argument as \ref{models-limits}, using \cite{HTT}*{Prop 5.3.3.3}, saying that filtered colimits commute with limits.
\end{proof}

Now we wish to study push-forwards of models, showing that taking left Kan extensions provides a left adjoint to the pullback functor. This will require some work; we subdivide it into two major parts.

We show that this is plausible:
\begin{prop}
\label{left-kan-preserves-prods}
Given a morphism $f:T\rightarrow U$ of theories, and a model $G:T\rightarrow\Spaces$, the left Kan extension of $G$ along $f$ preserves products and is thus a model of $U$.
\end{prop}
\begin{proof}
The left Kan extension is given by
\[(f_*G)(x)=\colim\left(T\timeso{U}U_{/x}\longrightarrow T\stackrel{G}{\longrightarrow}\Spaces\right).\]

We must show that $(f_*G)(1)\isom 1$ and $f_*(G)(x\times y)\isom f_*(G)(x)\times f_*(G)(y)$.

In both cases we show that there is a natural map from the colimit diagrams which define each side, which is cofinal (in the sense of Joyal, written up by Lurie \cite{HTT}*{4.1}), and thus there is an equivalence between them.

In the first case, we have
\[(f_*G)(1)=\colim\left(T\isom T\timeso{U}U_{/1}\longrightarrow T\longrightarrow\Spaces\right).\]

It is easy to see that the inclusion of the terminal object $(1,\id_1)$ into $T$ is cofinal. Indeed, by Joyal's characterisation of cofinal maps \cite{HTT}*{4.1.3.1}, we must show that $1\times_TT_{1/}$ is weakly contractible. This is clear: it has an initial object $1$.

Thus $1\rightarrow T$ induces an isomorphism of colimits. This terminal object is sent to $1\in T_0$ and thence to $1\in\Spaces_0$. This proves the first case.

In the second case, $f_*(G)(x\times y)$ is given by the colimit
\[(f_*G)(x\times y)=\colim\left(T\timeso{U}U_{/x\times y}\longrightarrow T\stackrel{G}{\longrightarrow}\Spaces\right).\]
There is a functor
\[\left(T\timeso{U}U_{/x}\right)\times\left(T\timeso{U}U_{/y}\right)\longrightarrow T\timeso{U}U_{/x\times y},\]
which sends
\[((t_1,f(t_1)\rightarrow x),(t_2,f(t_2)\rightarrow y))\longmapsto(t_1\times t_2,f(t_1\times t_2)\rightarrow x\times y)\]
in the evident way.

According to \cite{HTT}*{4.1.3.1}, to show this map is cofinal we need to show that, for any $(t,f(t)\rightarrow x\times y)\in\left(T\timeso{U}U_{/x\times y}\right)_0$, the simplicial set
\[\left(\left(T\timeso{U}U_{/x}\right)\times\left(T\timeso{U}U_{/y}\right)\right)\timeso{T\timeso{U}U_{/x\times y}}\left(T\timeso{U}U_{/x\times y}\right)_{(t,f(t)\rightarrow x\times y)/}\]
is weakly contractible.

This simplicial set is isomorphic to
\[\left(T^2\timeso{T}T_{t/}\right)\timeso{U^2\timeso{U}U_{f(t)/}}\left(\left(U_{/x}\times U_{/y}\right)\timeso{U_{/x\times y}}U_{f(t)//x\times y}\right),\]
which is the quasicategory of pairs $a,b\in T$ equipped with maps $t\rightarrow a\times b$, and 2-cells $f(t)\rightarrow f(a\times b)\rightarrow x\times y$.
But this quasicategory has an evident terminal object $\Delta:t\rightarrow t\times t$ and $f(t)\rightarrow f(t\times t)\rightarrow x\times y$, which makes it weakly contractible.

And the colimit of $\left(T\timeso{U}U_{/x}\right)\times\left(T\timeso{U}U_{/y}\right)$ is indeed $f_*(G)(x)\times f_*(G)(y)$, by Proposition \ref{products-and-colimits}. This completes the proof.
\end{proof}

Now we can finish the job:
\begin{prop}
\label{models-pushforward}
 Given a morphism $f:T\rightarrow U$ of theories, the pullback functor $f^*:\Mod(U)\rightarrow\Mod(T)$ has a left adjoint $f_*$, given by left Kan extension.
\end{prop}
\begin{proof}
We could do this simply by restricting the standard adjunction between $\Fun(T,\Spaces)$ and $\Fun(U,\Spaces)$ given by composition and left Kan extension. However, we build an adjunction by hand to make more of the structure visible.

First, we use $f^*$ to define a cocartesian fibration $\Mod(T/U)\rightarrow\Delta^1$, as described in subsection \ref{mapping-cylinders-quasicategories} (it is a cocartesian fibration, as proved in Proposition \ref{mapping-cylinders}).

We need to show that it is also cartesian, so it represents an adjunction. We have observed it to be an inner fibration already (in Proposition \ref{mapping-cylinders-1}; we just need to demonstrate the existence of cartesian lifts for edges. The simplicial set $\Delta^1$ only has one degenerate 1-cell; it is only over that cell that the problem is not vacuous.

Given $A\in\Mod(T)_0$, we take a left Kan extension of $A$ along $f$, given by
\[(f_*A)(x)=\colim\left(T\timeso{U}U_{/x}\longrightarrow T\stackrel{A}{\longrightarrow}\Spaces\right),\]
where we identify objects of $T$ and of $U$ for brevity. This is product-preserving by Proposition \ref{left-kan-preserves-prods}.

Our cartesian lift $\alpha$ shall have this as its zero vertex, so we must exhibit a morphism $f_*f^*A\rightarrow A$. This is provided by the universal property of the colimit.

We must now show that this 1-cell $\alpha$ from $f^*A$ to $A$ is cartesian. That means showing that the projection 
$\Mod(T/U)_{/\alpha}\rightarrow\Mod(T/U)\timeso{\Delta^1}\Delta^0$ is acyclic Kan.

Unpacking the definitions, this morphism is the evident projection
\[\Mod(U)_{/f_*A}\timeso{\Mod(T)_{/f^*f_*A}}\Mod(T)_{/(f^*f_*A\rightarrow A)} \longrightarrow \Mod(T)_{/A};\]
we can show that this is acyclic Kan by working pointwise and using the acyclic Kan condition of the colimit.
\end{proof}

Note that this gives us a notion of a free model $T(X)$ of a theory $T$ on a space $X$: a space can be viewed as a model of the initial theory $\Finop$, and we can use the push-forward associated to the initial morphism of theories $\Finop\rightarrow T$.

\subsection{Pointed theories}
\label{pointed-theories}

A theory is said to be \emph{pointed} if it has a zero object: an object $0$ which is both initial and terminal. This is standard categorical terminology, and is also justified by the following proposition:
\begin{prop}
If $T$ is a pointed theory, the terminal model $1$ (the model given by the constant $1$ functor) is a zero object in the category of models. In particular, any model $T\rightarrow\cU$ factors through $\cU_{1/}$. 
\end{prop}
\begin{proof}
Limits and colimits in $\Fun(T,\cU)$ are computed pointwise, so $1$ is initial and terminal there. The quasicategory $\Mod(T;\cU)=\Funpp(T,\cU)$ is just a $1$-full subcategory, so $1$ is still a zero object.
\end{proof}

It might be useful to have a quasicategory of pointed theories: this is the $1$-full subquasicategory of $\Theories$ whose objects are the pointed theories.

\begin{prop}
\label{finsop-property}
Let $\cU$ be any quasicategory with finite products. Then there is an equivalence \[\Funpp(\Finsop,\cU)\longrightarrow\cU_{1/}\]
between product-preserving functors from $\Finsop$ and $\cU$ and the quasicategory of objects of $\cU$ equipped with maps from the terminal object.
\end{prop}
\begin{proof}
The technique of proof is very similar to that of Proposition \ref{finop-is-trivial}.

We obtain a functor $\Funpp(\Finsop,\cU)\rightarrow\cU_{1/}$ by a slightly contorted process. First, choose a terminal object $1\in\cU$.

Consider the morphism $z:\Delta^1\rightarrow\Finsop$ given by the map of pointed sets $1_+\rightarrow 0_+$. This map induces a functor $I:\Funpp(\Finsop,\cU)\rightarrow\Fun(\Delta^1,\cU)$.

Since any product-preserving functor $\Finsop\rightarrow\cU$ sends $O_+$ to a terminal object, we can modify $I$ by an equivalence to land in the full subcategory of $\Fun(\Delta^1,\cU)$ on those objects sending the terminal edge to $1$.

By Proposition \ref{alternative-overcat}, we know that this is equivalent to $\cU_{1/}$. Combining these constructions gives us the required functor $\Funpp(\Finsop,\cU)\rightarrow\cU_{1/}$.

As in Proposition \ref{finop-is-trivial}, it will suffice to show that any product-preserving functors $\Finsop\rightarrow\cU$ are right Kan extensions of their restrictions along $z$ to $\Delta^1$.

Let $F:\Delta^1\rightarrow\cU$ be a morphism $1\rightarrow u$ in $\cU$.

Since our map $z$ is a full subcategory inclusion, we can calculate the right Kan extension $G$ of $F$ along $z$ as the limit
\[G(X_+)=\lim\left(\left((\Finsop)_{X_+/}\timeso{\Finsop}\Delta^1\right)\rightarrow \Delta^1\stackrel{F}{\rightarrow}\cU\right).\]
The category $(\Finsop)_{X_+/}\timeso{\Finsop}\Delta^1$ consists of $X$ discrete points (corresponding to the maps $1_+\rightarrow X_+$ for each element of $X$) and a copy of $\Delta^1$ (corresponding to the zero maps $1_+\rightarrow X_+$ and $0_+\rightarrow X_+$).

Thus the limit is $u^X\times 1$, as required.
\end{proof}

In cases where the terminal object is also initial, this says that every object extends to a map from $\Finsop$.
\begin{corol}
Suppose $\cU$ has finite products and a zero object. Then \[\Funpp(\Finsop,\cU)\isom\cU.\]
\end{corol}

This allows us to give a valuable structure result:
\begin{prop}
\label{finsop-is-universal}
The theory $\Finsop$ is the initial pointed theory. Moreover, a pointed theory can be regarded as a quasicategory with finite products $T$ and an essentially surjective, product-preserving functor $\Finsop\rightarrow T$.
\end{prop}

\subsection{Structure on algebraic theories}

In this section we show that the quasicategory of theories is complete (Proposition \ref{theories-complete}), which is straightforward, and that it is cocomplete (Proposition \ref{theories-cocomplete}), which is much harder.

In order to prove the latter result, we introduce a good deal of machinery. Intrinsic in this machinery is the ability to take the free theory on some fairly general collection of data, but we apply it only to take the free theory on a diagram consisting of other theories. Thus we anticipate that the methods introduced here could be used to prove other theorems of this general character.
\begin{prop}
\label{theories-complete}
 The quasicategory of theories is complete.
\end{prop}
\begin{proof}
 An $I$-shaped diagram in theories yields an underlying diagram $F:I\rightarrow\Cinfty$. We consider also the functors from the initial  theory $\Finop$; for each $i$, we write $p(i)$ for the map $\Finop\rightarrow F(i)$. This is classified by a Cartesian fibration $X\rightarrow I^\op$. Lurie's model, which we recall from \cite{HTT}*{3.3.3}, for the limit of this diagram (in $\Cinfty$) is the quasicategory of Cartesian sections $I^\op\rightarrow X$ (that is, the quasicategory of sections which take 1-cells to Cartesian 1-cells).

 We consider the full subquasicategory of this on the objects $s:I^\op\rightarrow X$ for which there is a finite set $A$ such that $s(i)=p(i)(A)$, that is, those which act diagonally on objects.

 Any cone over $F$ in $\Theories$ acts diagonally on the objects, up to equivalence, since the maps commute with the structure maps. Hence the universal property of the product in $\Cinfty$ gives us a universal property for this subobject in $\Theories$.
\end{proof}

Now we turn our attention to showing that theories have all colimits. This will require some technical work, and we build up to the proof slowly.

The plan is as follows: Lurie has proved that the quasicategory of quasicategories is cocomplete. Thus, for any simplicial set $D$, the colimit $\colim_{\Theories}(D)$ of a diagram in theories factors uniquely through the colimit $\colim_{\Cinfty}(D)$ in the quasicategory of quasicategories. Indeed, we should expect it to be the universal quasicategory with a functor from $\colim_{\Cinfty}(D)$ such that the images of all the product cones in elements of $D$ are product cones.

Consider the quasicategory $(\Cinfty)_{(1\star D)/}$ of quasicategories with a map from $1\star D$. We are interested in the full subcategory $(\Cinfty)^{\lim}_{(1\star D)/}$ with objects the quasicategorical limit cones $(1\star D)\rightarrow\cC$.

Our first step is this:
\begin{prop}
\label{inclusion-of-limits-into-cones}
The inclusion functor
\[F:(\Cinfty)^{\lim}_{(1\star D)/}\longrightarrow(\Cinfty)_{(1\star D)/}\]
preserves all limits.
\end{prop}
\begin{proof}
By \cite{HTT}*{4.4.2.6}, it suffices to show it preserves all products and pullbacks.

In this proof we write $*$ for the terminal simplicial set, to avoid overuse of the symbol $1$.

Given a set of quasicategories and maps $\{*\star D\rightarrow \cC_\alpha\}_{\alpha\in A}$, all of them limit cones, then the diagonal map $(*\star D)\rightarrow\prod_{\alpha\in A}\cC_\alpha$ can easily be shown to be a limit cone.

Now, we have to deal with pullbacks of quasicategories; we recall the setup of Definition \ref{homotopy-pullbacks-spaces}.

Now, suppose we have limit cones $(*\star D)\rightarrow\cC_1,\cC_2,\cE$. We then have a diagonal map $(*\star D)\rightarrow\cC$, and must show that this too is a limit cone. Suppose we have a cofibration $I\rightarrow J$; we must show that there are liftings
\begin{displaymath}
\xymatrix{I\ar[r]\ar[d]&J\ar[d]\dar[dl]\\
          \cC_{/(*\star D)}\ar[r]&\cC_{/D},}
\end{displaymath}
or equivalently that there are extensions
\begin{displaymath}
\xymatrix{(I\star *\cup J\star\emptyset)\star D\ar[r]\ar[dr]&(J\star *\star D)\dar[d]\\
          &\cC,}
\end{displaymath}
provided that the restruction to $*\star D$ is the given cone.

Unravelling using the definition of $\cC$, we are demanding extensions
\begin{displaymath}
\xymatrix{
(I\star *\cup J\star\emptyset)\star D\ar[rr]\ar[dd]\ar[dr]&&\cC_1\ar[dd]\\
&J\star *\star D\dar[ur]\ar[dd]&\\
((E2\times I)\star *\cup(E2\times J)\star\emptyset)\star D\ar[rr]\ar[dr]&&\cE\\
&(E2\times J)\star *\star D\dar[ur]&\\
(I\star *\cup J\star\emptyset)\star D\ar[rr]\ar[dr]\ar[uu]&&\cC_2\ar[uu]\\
&J\star *\star D\dar[ur]\ar[uu]&}
\end{displaymath}
We can extend the top and bottom without difficulty, using that the maps $(*\star D)\rightarrow\cC_1,\cC_2$ are limit cones. This leaves us with an extension problem
\begin{displaymath}
\xymatrix{
\left((E2\times I)\star *\cup(E2\times J)\star\emptyset\cup (\{0,1\}\times J)\star *\right)\star D\ar[r]\ar[dr]&(E2\times J)\star *\star D\dar[d]\\
&\cE}
\end{displaymath}
which is readily checked to be a right lifting against a cofibration, and so follows from the fact that $(*\star D)\rightarrow\cE$ is a limit cone.

This completes the proof.
\end{proof}

Now we consider the diagram
\begin{displaymath}
\xymatrix{(\Cinfty)^{\lim}_{(1\star D)/}\ar[r]^F\ar[dr]_{\sim}&(\Cinfty)_{(1\star D)/}\ar[d]\\
&(\Cinfty)_{D/}.}
\end{displaymath}
The diagonal map is an acyclic Kan fibration, since every diagram naturally has a contractible space of limits.

Note that $\Cinfty$ is a presentable category \cite{LurieBicat}*{Remark 1.2.11}, and \cite{HTT}*{5.5.3.11} shows that undercategories of presentable categories are presentable. Thus all the categories in the diagram are presentable.

Also, the proof of Proposition \ref{inclusion-of-limits-into-cones} demonstrates that colimits in $(\Cinfty)_{D/}$ and $(\Cinfty)_{(1\star D)/}$ are computed in $\Cinfty$, and thus (using that the diagonal map is an equivalence) all three functors preserve colimits.

Accordingly, we can apply Lurie's Adjoint Functor Theorem \cite{HTT}*{5.5.2.9} to show the following:
\begin{prop}
\label{forcing-a-limit}
The functor $F:(\Cinfty)^{\lim}_{(1\star D)/}\longrightarrow(\Cinfty)_{(1\star D)/}$ admits a left and a right adjoint.
\end{prop}

A straightforward consequence of the existence of a left adjoint is that, for every quasicategory $\cC$ and map $1\star D\rightarrow\cC$, there is a universal quasicategory $\cC\rightarrow\cC'$ such that the composite $1\star D\rightarrow\cC'$ is a limit cone, in the sense that
\[(\Cinfty)_{(1\star D\rightarrow\cC\rightarrow\cC')/}\timeso{(\Cinfty)_{(1\star D\rightarrow\cC)/}}(\Cinfty)^{\lim}_{(1\star D\rightarrow\cC)/}\longrightarrow(\Cinfty)_{(1\star D\rightarrow\cC)/}\]
is acyclic Kan.

Indeed, the morphism $\cC\rightarrow\cC'$ is just the unit of the adjunction.

Now, suppose we have a diagram $K\rightarrow\Theories$. We will construct a colimit. Firstly, the extension $K\rightarrow\Cinfty$ has a colimit $K\star 1\rightarrow\Cinfty$. Transfinitely enumerate the finite product diagrams as $\{f_\alpha:(1\star X_\alpha)\rightarrow K_{s(\alpha)}\}$. With this notation, we prove the result we were aiming for:

\begin{prop}
\label{theories-cocomplete}
The quasicategory of theories is cocomplete.
\end{prop}
\begin{proof}
We provide a colimit for any diagram $F:K\rightarrow\Theories$.

Firstly, we can obtain from our diagram $F$ a diagram $F':1\star K\rightarrow\Cinfty$, sending $1$ to the initial theory $\Finop$. We take the colimit of that, using \cite{HTT}*{3.3.4}. We claim that the resulting colimit cocone has essentially surjective structure maps.

 Indeed, any object in $\colim(F')$ is in the essential image of $F'(z)$ for some $z\in 1\star K$: the structure maps are jointly essentially surjective. (To prove this, it is quick to verify that the essential image of $1\star K$ in $\colim(F')$ satisfies the colimit property, and is thus all of it).

 However, any $a\in F'(z)_0$ is the essential image of some $A\in(\Finop)_0$. Also, there is an equivalence (induced by the image of the 2-cell $(1\star\{z\}\star 1)$ in $\Cinfty$) between the image of $a$ and the image of $A$ in $\colim(F')$. So every structure map has the same essential image: they're all essentially surjective.

Now, we will manufacture a colimit in $\Theories$.

We start with $X_0=\colim_{\Cinfty}(F')$. We can transfinitely enumerate the finite product diagrams in the quasicategory $\Finop$ as
\[\{f_\alpha:(1\star D_\alpha)\rightarrow\Finop\}_{\alpha<\kappa}\]
for some ordinal $\kappa$, where $D_\alpha$ is a discrete simplicial set, and $z_\alpha\in K_0$. We choose to do this with redundancy: we want each individual product diagram to appear infinitely many times and be cofinal in $\kappa$. One straightforward way to ensure this is to enumerate them without repetition with ordertype $\lambda$, then take $\kappa=\lambda\omega$ and repeat our list $\omega$-many times.

Our aim is to produce $\colim_{\Theories}(F')$ by starting from $X_0$ and extending all the maps $f_\alpha$ to limit cones.

Proposition \ref{forcing-a-limit} supplies us with a quasicategory $X_1$ and a unit map $u$ such that the composite
\[(1\star D_0)\longrightarrow X_0\stackrel{u}{\longrightarrow}X_1\]
is a limit cone.

Similarly, we produce $X_2$ from $X_1$ by using the adjunction to provide a quasicategory from which the map from $1\star D_1$ is a limit cone. Then we proceed by transfinite induction, extending to limit ordinals by taking the filtered colimits (in $\Cinfty$) of the preceding quasicategories:
\[X_{\lim_i(\alpha_i)}=\lim_i(X_{\alpha_i}).\]

The resulting quasicategory $X_\kappa$ is our colimit. We have several checks to make to show this to be the case.

Firstly, it is necessary to show we haven't enlarged our quasicategory in an unacceptable manner:
\begin{claim}
All the structure maps $X_\alpha\rightarrow X_\beta$ for $\alpha<\beta$ are essentially surjective.
\end{claim}
\begin{innerproof}{of claim}
It suffices to show both that the successor maps $X_\alpha\rightarrow X_{\alpha+1}$ are essentially surjective, and also that a colimit of shape $\omega$ of essentially surjective maps is essentially surjective.

A similar argument works for both. We can show for each that a failure to be essentially surjective would violate the universal property: that the image of the morphism would provide a smaller object with the same property.

Indeed, if $X_\alpha\rightarrow X_{\alpha+1}$ was not essentially surjective, the image $X'_{\alpha+1}\subset X_{\alpha+1}$ would result in the nontrivial factorisation
\[(1\star D_\alpha)\longrightarrow X'_{\alpha+1}\longrightarrow X_{\alpha+1},\]
and the left-hand map can easily be checked to be a product cone. This contradicts the universal property of the adjunction.

Similarly, if a colimit
\[Z_0\longrightarrow Z_1\longrightarrow Z_2\longrightarrow\cdots\longrightarrow Z_\omega\]
of essentially surjective maps of quasicategories is not essentially surjective, then the essential image factors the structure maps of the colimit nontrivially, which contradicts the universal property of the colimit.
\end{innerproof}

Secondly, we need to show it is indeed a theory, and that the structure maps we've defined are maps of theories. We've done essential surjectivity already, so we just need the following:
\begin{claim}
The defined maps $\Finop\rightarrow X_\kappa$ preserve all finite products.
\end{claim}
\begin{innerproof}{of claim}
Given a limit cone $1\star D\rightarrow\Finop$, we must show that the composite $1\star D\rightarrow X_\kappa$ is a limit cone too. Given a lifting problem for a cofibration $\partial\Delta^n\rightarrow\Delta^n$ as follows:
\begin{displaymath}
\xymatrix{\partial\Delta^n\ar[r]\ar[d]&\Delta^n\ar[d]\dar[dl]\\
          (X_\kappa)_{/(1\star D)}\ar[r]&(X_\kappa)_{/D},}
\end{displaymath}
we rewrite it as
\begin{displaymath}
\xymatrix{(\partial\Delta^n\star 1\star D)\cup(\Delta^n\star\emptyset\star D)\ar[r]\ar[dr]&\Delta^n\star 1\star D\dar[d]\\
          &X_\kappa.}
\end{displaymath}
Since the simplicial set $(\partial\Delta^n\star 1\star D)\cup(\Delta^n\star\emptyset\star D)$ is finite, the map from it to $X_\kappa$ factors through some $X_\lambda$ for which $1\star D\rightarrow X_\lambda$ is a product cone (since, by construction, the set of such ordinals $\lambda$ is cofinal in $\kappa$).

The required extension exists in that $X_\lambda$ and thus also in $X_\kappa$.
\end{innerproof}

Lastly, of course, we need to verify the universal property of a colimit.
\begin{claim}
$X_\kappa$ is universal among theories under $F$.
\end{claim}
\begin{innerproof}{of claim}
We need to show that the functor $\Theories_{(F\star 1)/}\rightarrow\Theories_{F/}$ is acyclic Kan. Suppose given a cofibration $I\rightarrow J$; we have a lifting problem
\begin{displaymath}
\xymatrix{I\ar[r]\ar[d]&J\ar[d]\\
\Theories_{(F\star 1)/}\ar[r]&\Theories_{F/}.}
\end{displaymath}

It suffices to consider cofibrations $\partial\Delta^n\rightarrow\Delta^n$. We consider the $n=0$ and $n\geq 1$ separately.

If $n=0$, our cofibration is $\emptyset\rightarrow 1$: we have a cone $K\star 1\rightarrow\Theories$ describing a theory $T$ under $F$; the aim is to factor it through $X_\kappa$.

Since $X_0$ is the colimit of $F$ in $\Cinfty$, we have a diagram $K\star 1\star 1\rightarrow\Cinfty$ factoring our cone through $X_0$. Working under $F$, since the maps $F(x)\rightarrow T$ are product-preserving, we can factorise this successively through the $X_\lambda$ to get an essentially surjective, product-preserving map $X_\kappa\rightarrow T$ under $F$ as required.

Now, in case $n\geq 1$, we have compatible functors $K\star 1\star\partial\Delta^n\rightarrow\Theories$ and $K\star\emptyset\star\Delta^n\rightarrow\Theories$, with the middle $1$ sent to $X_\kappa$. Equivalently, this is a diagram $K\star\partial\Delta^{1+n}\rightarrow\Theories$ and we need to extend it to $K\star\Delta^{1+n}\rightarrow\Theories$.

Since $X_0$ is the colimit of $F$, we can extend this the underlying diagram $K\star\Delta^{1+n}\rightarrow\Cinfty$ to a diagram $K\star 1\star\Delta^{1+n}\rightarrow\Cinfty$, with the middle $1$ sent to $X_0$. Using the universal property of the adjunction and the colimiting property, we can extend this to a map $K\star N(\kappa+1)\star\Delta^{1+n}\rightarrow\Cinfty$, where $N(\kappa+1)$ denotes the nerve of the ordinal $\kappa+1$ viewed as a poset (that is, as the poset of ordinals less than or equal to $\kappa$), and where the ordinal $\lambda$ is sent to $X_\lambda$.

The terminal vertex of $N(\kappa+1)$ and the initial vertex of $\Delta^{1+n}$ are both sent to $X_\kappa$. Moreover, by construction, they are identical under $F$, and so by construction the edge between them is the identity. Restriction to $K\star\{\kappa\}\star\Delta^n\rightarrow\Cinfty$ thus gives us the required diagram in $\Cinfty$; since all the edges were present already and were morphisms of $\Theories$, this is also a diagram in $\Theories$.
\end{innerproof}
This completes the proof.
\end{proof}

We can do similar things with this method:
\begin{prop}
\label{cinftypp-cocomplete}
The quasicategory $\Cinftypp$ of quasicategories with all finite products, and product-preserving functors between them, has colimits.
\end{prop}
\begin{proof}
The proof of the preceding proposition generalises in a straightforward manner. Since our diagrams are no longer cones of essentially surjective maps under $\Finop$, we need to consider products in all the diagrams and force their images to all be products (whereas before it sufficed to consider only those in $\Finop$). We no longer need to ensure essential surjectivity, but we do however need to provide limits for all the new objects introduced. So we intersperse the operations which force cones to be product cones with operations that adjoin new products for the objects (using the methods of \cite{HTT}*{5.3.6}).
\end{proof}

\begin{comment}
\begin{remark}
  The construction of colimits given by the proof above depends on the fact that $\Cinfty$ is the $(\infty,1)$-category of $(\infty,1)$-categories, functors and homotopies between them.

  If we were taking 2-colimits in the $(\infty,2)$-category of $(\infty,1)$-categories, functors, natural transformations, and homotopies between them, alternative methods would be needed. Indeed, a 2-colimit of quasicategories would be some kind of Grothendieck construction: in particular, the structure maps of the 2-colimit 2-cocone would be very far from being essentially surjective.
\end{remark}
\end{comment}

\subsection{Free models on sets}

Free models for a theory are shown to exist by Proposition \ref{models-pushforward}; this section records a more explicit, less involved construction of free models on finite sets.

Let $T$ be a theory. We suppose given a functorial model $\Map_T(-,-):T^\op\times T\rightarrow\Spaces$ for the homspaces in $T$. Such models are shown to exist and are discussed further in \cite{HTT}*{Section 1.2.2}.

The map $\Finop\rightarrow T$ from the initial theory is equivalent to a functor $\Fin\rightarrow T^\op$, and we can compose this with the homspace functor to get a map
\[\Free:\Fin\rightarrow\Fun(T,\Spaces).\]
In other words, we take $\Free_X(Y)=\Map_T(X,Y)$.

The functor $\Free_X$ is product-preserving since $\Map_T(X,-)$ is, so we actually get a functor
\[\Free:\Fin\rightarrow\Mod(T).\]

This behaves as we would hope:
\begin{prop}
The functor $\Free_X$ is indeed the free model of $T$ on $X$.
\end{prop}
\begin{proof}
For any model $A$ of $T$, we have a natural equivalence
\[\Spaces(X,A(1))\isom\Mod(T)(\Free_X,A);\]
This is a straightforward exercise using the quasicategorical Yoneda lemma of \cite{HTT}*{Section 5.1}.
\end{proof}
In particular, this agrees with the more general construction of Proposition~\ref{models-pushforward}.

We can also prove:
\begin{prop}
A theory $T$ is equivalent to the opposite of the full subquasicategory of $\Mod(T)$ on the free models on finite sets.
\end{prop}
\begin{proof}
The Yoneda embedding used above is full and faithful; and the functor is evidently essentially surjective on objects.
\end{proof}

\section{Span diagrams for studying monoids}

\subsection{The bicategory $\TwoSpan$}
\label{span-bicategory}

Throughout this subsection we assume given a canonical, functorial choice of pullbacks of finite sets.

We introduce a bicategory $\TwoSpan$ of span diagrams. A 0-cell of $\TwoSpan$ is a finite set. A 1-cell from $X_0$ to $X_1$ is a span diagram $X_0\leftarrow Y\rightarrow X_1$ of finite sets. A 2-cell between diagrams $X_0\leftarrow Y\rightarrow X_1$ and $X_0\leftarrow Y'\rightarrow X_1$ is an isomorphism $f:Y\stackrel{\sim}{\rightarrow}Y'$ fitting into a diagram as follows:
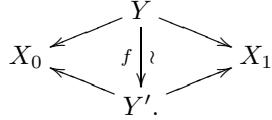
\begin{displaymath}
\xymatrix@R-20pt{&Y\ar[dl]\ar[dr]\ar[dd]_f^{\wr}&\\
          X_0&&X_1\\
          &Y'.\ar[ul]\ar[ur]}
\end{displaymath}
2-cells compose in the obvious way; 1-cells compose by taking pullbacks: the composite of $X_0\leftarrow X_{01}\rightarrow X_1$ and $X_1\leftarrow X_{12}\rightarrow X_2$ is given by $X_0\leftarrow X_{02}\rightarrow X_2$, where $X_{02}$ is the following pullback:
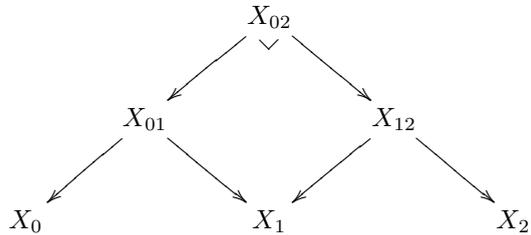
\begin{displaymath}
\xymatrix{&&X_{02}\pb{270}\ar[dl]\ar[dr]&&\\
          &X_{01}\ar[dl]\ar[dr]&&X_{12}\ar[dl]\ar[dr]&\\
          X_0&&X_1&&X_2.}
\end{displaymath}
It is a simple exercise to show that this gives a bicategory.

We will later have cause to use a generalisation of this notion. Given a category $\cC$ which has all pullbacks, and a functorial choice of pullbacks, we can define the bicategory $\TwoSpan(\cC)$ of spans in $\cC$: 0-cells are objects of $\cC$, 1-cells are span diagrams in $\cC$, and 2-cells are isomorphisms of spans. So, in this notation, our category $\TwoSpan$ is $\TwoSpan(\Fin)$.

\subsection{Equivalences in $\TwoSpan$}

We work with the weak 2-category of spans $\TwoSpan(\cC)$, where $\cC$ is any category with pullbacks. First we prove a more-or-less standard lemma of ordinary category theory:

\begin{prop}
\label{pb-se}
Pullbacks of split epimorphisms are split epimorphisms.
\end{prop}
\begin{proof}
Suppose given a diagram
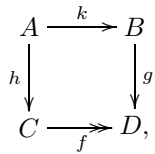
\begin{displaymath}
 \xymatrix{A\ar[r]^k  \ar[d]_h&B\ar[d]^g\\
           C\sear[r]_f        &D,}
\end{displaymath}
where the bottom morphism $f$ is a split epimorphism: a morphism such that there is $f':D\rightarrow C$ with $ff'=1_D$.

This affords us a map $f'g:B\rightarrow C$. Now, we have $f(f'g) = g1_B$, and so, by the definition of the pullback, there is a map $k':B\rightarrow A$ with $kk'=1_B$, as required.
\end{proof}

This allows us to prove an important structural result for span categories:

\begin{prop}
\label{span-isos}
Objects $X,Y\in\Ob\TwoSpan(\cC)$ are equivalent if and only if they are isomorphic as objects of $\cC$.
\end{prop}
\begin{proof}
Given two isomorphic objects in $\cC$, any span of isomorphisms between them forms an equivalence in $\TwoSpan(\cC)$.

The data of an equivalence consists of 1-cells $X\stackrel{a}{\leftarrow}U\stackrel{b}{\rightarrow}Y$ and $Y\stackrel{c}{\leftarrow}V\stackrel{d}{\rightarrow}Y$, fitting into diagrams
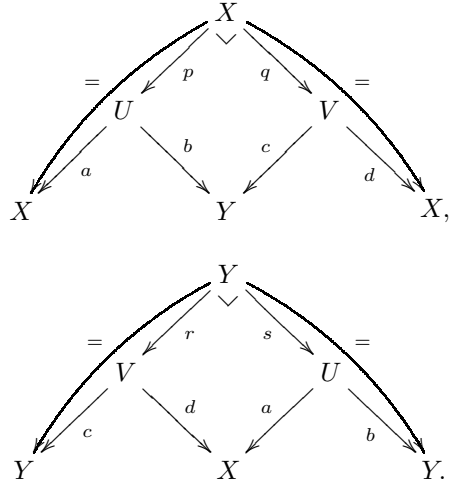
\begin{displaymath}
 \xymatrix{&&X\pb{270}\ar[dl]^p\ar[dr]_q\ar@/_2ex/[ddll]_{=}\ar@/^2ex/[ddrr]^{=}&&\\
           &U\ar[dl]^{a}\ar[dr]^{b}&&V\ar[dl]_{c}\ar[dr]_{d}&\\
           X&&Y&&X,}
\end{displaymath}
\begin{displaymath}
 \xymatrix{&&Y\pb{270}\ar[dl]^r\ar[dr]_s\ar@/_2ex/[ddll]_{=}\ar@/^2ex/[ddrr]^{=}&&\\
           &V\ar[dl]^{c}\ar[dr]^{d}&&U\ar[dl]_{a}\ar[dr]_{b}&\\
           Y&&X&&Y.}
\end{displaymath}
The maps $a$, $b$, $c$ and $d$ are split epimorphisms (by inspecting the left and right composites in each diagram). But this means that $p$, $q$, $r$ and $s$ are also split epimorphisms, by Proposition \ref{pb-se}.

However $p$, $q$, $r$ and $s$ are also split monomorphisms (by inspection of the left and right composites), and thus isomorphisms (since if a morphism is left and right invertible, the inverses agree). This clearly means that $a$, $b$, $c$ and $d$ are isomorphisms, and thus that $X$ and $Y$ are isomorphic.
\end{proof}

Continuing the analysis, we have the following proposition:
\begin{prop}
\label{span-auts}
Let $\cC$ be a category. The group of isomorphism classes of automorphisms of $X$ in $\TwoSpan(\cC)$ is isomorphic to $\Aut_\cC(X)$.
\end{prop}
\begin{proof}
By \ref{span-isos}, any automorphism of $X$ in $\TwoSpan(\cC)$ looks like
\[X\stackrel{\sim}{\longleftarrow}X'\stackrel{\sim}{\longrightarrow}X.\]
But such a span diagram is uniquely isomorphic to exactly one of the form
\[X\stackrel{=}{\longleftarrow}X\stackrel{\sim}{\longrightarrow}X;\]
this proves the claim.
\end{proof}

\subsection{The $\Span$ quasicategory}
\label{span-quasicategory}

We now define a quasicategory $\Span$, one of the principal objects of study of this thesis, which is isomorphic to the nerve of the bicategory $\TwoSpan$.

Define $C_n$ to be the poset of nonempty subintervals $(i,i+1,\ldots,j)$ in $[n]=(0,\ldots,n)$, equipped with the \emph{reverse} inclusion ordering. We regard $C_n$ as a category.

The poset of nonempty subintervals of a totally ordered set is a functorial construction, so the collection $C=\{C_n\}$ forms a cosimplicial object in categories as we vary over all finite totally ordered sets $(0,\ldots,n)$.

This enables us to define a simplicial set, which we shall soon prove (in Proposition \ref{Span-structure}) to be a quasicategory:
\begin{defn}
Let $\cC$ be an ordinary category with pullbacks. We define the \emph{span quasicategory} $\Span(\cC)$ to be the simplicial set whose $n$-cells $\Span_n$ are the collection of functors $F:C_n\rightarrow\cC$ from $C_n$ to the category $\cC$, with the condition that, if $I$ and $J$ are two nonempty intervals in $[n]$ with nonempty intersection, the diagram
\begin{displaymath}
\xymatrix{
F(I\cup J)\pb{315}\ar[r]\ar[d]&F(I)\ar[d]\\
F(J)\ar[r]&F(I\cap J)}
\end{displaymath}
is a pullback.
\end{defn}
We refer to this condition later as the \emph{pullback property}. The collection $\Span(\cC)$ is indeed a simplicial set, since $C$ is a cosimplicial category, and taking faces and degeneracies preserves the pullback property.

If $Y:C_n\rightarrow\cC$ is an $n$-cell of the quasicategory $\Span(\cC)$, then we will write $Y_{ij}$ for $Y((i,\ldots,j))$ and $Y_i$ for $Y((i))$. If $i\leq i'\leq j'\leq j$, then we write $Y_{ij\rightarrow i'j'}$ for the structure map $Y_{ij}\rightarrow Y_{i'j'}$ induced by the inclusion.

Now, we have our formal statement:
\begin{prop}
\label{Span-structure}
Suppose $\cC$ is any ordinary category with pullbacks. Then we have an isomorphism of simplicial sets $N(\TwoSpan(\cC))\isom\Span(\cC)$. Thus $\Span(\cC)$ is a $(2,1)$-category and in particular (as suggested in the definition above) a quasicategory.
\end{prop}
\begin{proof}
We refer back to Section \ref{two-one-categories} for notation on bicategories. Given an $n$-cell $\{X_i,f_{ij},\theta_{ijk}\}\in N(\TwoSpan(\cC))_n$, we associate an $n$-cell $Y\in\Span(\cC)_n$.

We take $Y_i=X_i$ for all $i$. Further, we take $Y_{ij}$ to be the middle part of the span 1-cell given by $f_{ij}$, so we have a diagram $Y_i\leftarrow Y_{ij}\rightarrow Y_j$ for all $i<j$.

What is more, the 2-cell $\theta_{ijk}$ gives us a diagram as follows:
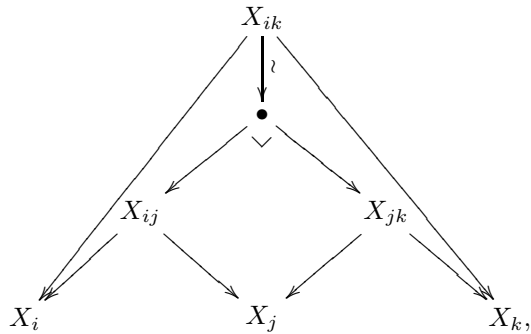
\begin{displaymath}
\xymatrix{&&X_{ik}\ar[dddll]\ar[dddrr]\ar[d]^{\wr}&&\\
          &&\bullet\pb{270}\ar[dl]\ar[dr]&&\\
          &X_{ij}\ar[dl]\ar[dr]&&X_{jk}\ar[dl]\ar[dr]&\\
          X_i&&X_j&&X_k,}
\end{displaymath}
for every $i<j<k$.

However, such diagrams are in 1-1 correspondence with diagrams
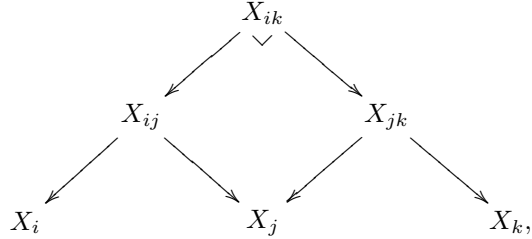
\begin{displaymath}
\xymatrix{&&X_{ik}\pb{270}\ar[dl]\ar[dr]&&\\
          &X_{ij}\ar[dl]\ar[dr]&&X_{jk}\ar[dl]\ar[dr]&\\
          X_i&&X_j&&X_k,}
\end{displaymath}
where the composite $X_i\leftarrow X_{ik}\rightarrow X_k$ is the given span diagram for $i<k$.

The compatibility condition gives all the other pullbacks, and the functoriality of the maps $X_{ij}\rightarrow X_{ij'}$ and $X_{ij}\rightarrow X_{i'j}$. 

This construction is reversible (and naturally commutes with faces and degeneracies) so we get an isomorphism of simplicial sets.
\end{proof}

In a similar fashion, there are functors $\bar L:\cC\rightarrow\Span(\cC)$ and $\bar R:\cC^\op\rightarrow\Span(\cC)$, for any category $\cC$. They are evidently faithful, and according to Proposition \ref{span-isos}, if $\bar L(f)$ or $\bar R(f)$ is an equivalence then $f$ is an isomorphism.

We move on to considering products in the quasicategory $\Span$.

\begin{prop}
The quasicategory $\Span$ has finite products. The product of objects $A$ and $B$ is $A\sqcup B$.
\end{prop}
\begin{proof}
Recall the definition of limits in quasicategories: if $f:K\rightarrow\Span$ is a morphism of simplicial sets, then a limit of $f$ is a terminal object of the over-category $\Span_{/f}$, given by
\[\left(\Span_{/f}\right)_n=\left\{\text{maps $\Delta_n\star K\rightarrow\Span$ which extend $f$}\right\}.\]

For us, $K=2=\{0,1\}$, with $f(0)=A$ and $f(1)=B$. Thus
\begin{eqnarray*}
\left(\Span_{/f}\right)_n&=&\left\{\text{maps $\Delta_n\star 2\rightarrow\Span$ which extend $f$}\right\}\\
&=&\big\{(X,Y)\in\Span_{n+1}^2\vert
             d_n X=d_n Y,\\
& &\hskip 25mm d_0d_1\cdots d_{n-1}X=A,\\
& &\hskip 25mm d_0d_1\cdots d_{n-1}Y=B\big\}\\
&=&\big\{\text{$X,Y:C_{n+1}\rightarrow\Fin$ with pullback property, such that}\\
& &\hskip 7mm\text{$X|_{C_n}=Y|_{C_n}$, $X(n+1)=A$ and $Y(n+1)=B$.}\big\}
\end{eqnarray*}

We now specify the object $P$ of $(\Span_{/f})_0$ which we claim is the product: it consists of the object $A\sqcup B\in\Span_0$, with projection maps $\{A\sqcup B\leftarrow A\rightarrow A\}$ and $\{A\sqcup B\leftarrow B\rightarrow B\}$.

We need to show that it is a strongly final object in $\Span_{/f}$. This means showing that any diagram $F:\partial\Delta^n\rightarrow\Span_{/f}$ with $F(n)=P$ extends to a diagram $\Delta^n\rightarrow\Span_{/f}$. This will be a straightforward, but notationally heavy, check.

Define $\hat C'_{n+1}$ to be the poset of subintervals of $\{0,\ldots,n+1\}$ that do not contain all of $\{0,\ldots,n\}$ (with the reverse inclusion order). Restricting to $\{0,\ldots,n\}$, we recover the poset $\hat C_n$ of proper subintervals of $\{0,\ldots,n\}$ defined in Section \ref{span-quasicategory}.

The simplicial structure on $\Delta^n$ guarantees that maps $\partial\Delta^n\rightarrow\Span_{/f}$ assemble to form diagrams $X,Y:\hat C'_{n+1}\rightarrow\Fin$ with the pullback property, such that $X|_{\hat C_n}=Y|_{\hat C_n}$, $X(n)=Y(n)=A\sqcup B$, $X(n,n+1)=X(n+1)=A$ and $Y(n,n+1)=Y(n+1)=B$.

For example, if $n=3$ the diagram $X$ is as follows:
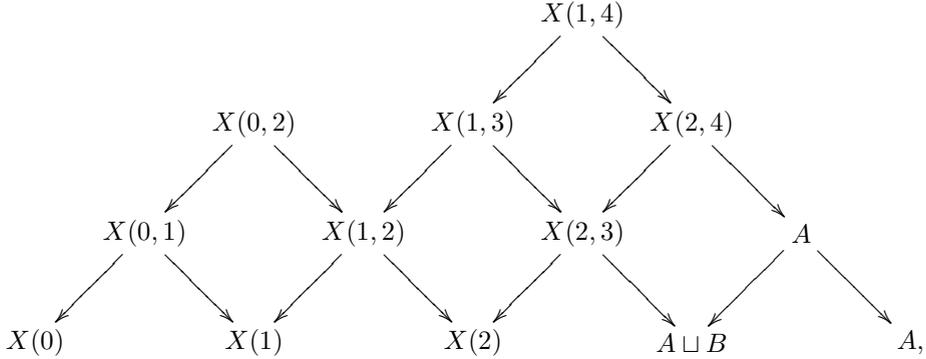
\begin{displaymath}
\xymatrix@!R=6mm@!C=6mm{
&&&&&X(1,4)\ar[dl]\ar[dr]&&&\\
&&X(0,2)\ar[dl]\ar[dr]&&X(1,3)\ar[dl]\ar[dr]&&X(2,4)\ar[dl]\ar[dr]&&\\
&X(0,1)\ar[dl]\ar[dr]&&X(1,2)\ar[dl]\ar[dr]&&X(2,3)\ar[dl]\ar[dr]&&A\ar[dl]\ar[dr]&\\
X(0)&&X(1)&&X(2)&&A\sqcup B&&A,}
\end{displaymath}
and the diagram $Y$ is similar.

We can extend these to $C_{n+1}$ by defining
\begin{eqnarray*}
X(0,n)  &=&{\lim}_{\hat C_n}X,\\
Y(0,n)  &=&{\lim}_{\hat C_n}Y,\\
X(0,n+1)&=&{\lim}_{\hat C'_n}X,\\
Y(0,n+1)&=&{\lim}_{\hat C'_n}Y.
\end{eqnarray*}
We clearly have $X|_{C_n}=Y|_{C_n}$, have $X(n+1)=A$ and $Y(n+1)=B$ by definition, and it is quick to check the pullback property.
\end{proof}

The same proof suffices to prove the following:
\begin{prop}
\label{prod-span-c}
For any category $\cC$ with finite coproducts and finite limits, finite products in $\Span(\cC)$ exist, and are given on objects by coproducts in $\cC$. The inclusion maps are defined analogously to the case $\cC=\Fin$ above.
\end{prop}

As an important corollary, we have:
\begin{prop}
The functor $R$ makes the category $\Span$ into an algebraic theory, as introduced in Definition \ref{theory-defn}.
\end{prop}
Accordingly, since $\Span$ was motivated by the desire to produce a quasicategorical version of the theory of monoids, we define:
\begin{defn}
\label{lawvere-monoid-object}
Let $\cC$ be a quasicategory with finite products. A \emph{(Lawvere) monoid object} in $\cC$ is a model of $\Span$ in $\cC$: a product-preserving functor $\Span\rightarrow\cC$.
\end{defn}

Also, since $\Span$ is self-opposite, we have
\begin{prop}
The category $\Span(\cC)$ has coproducts, which agree with products.
\end{prop}
An immediate consequence of this is that the theory $\Span$ is pointed, as defined in Section \ref{pointed-theories}; various important consequences of this are given there too.

\begin{remark}
As we have seen, the category $\Finsop$ is equivalent to the quasicategory of spans whose right arm is a monomorphism. Proposition \ref{finsop-property} tells us that models of $\Finsop$ are in some sense pointed objects. So $\Finsop$ carries the part of the theory of monoids dealing with the unit, but not the product.

If a theory is required which has an associative but nonunital product, the natural choice is the quasicategory of spans whose right arm is an epimorphism; this ensures that products are only taken over nonempty sets.
\end{remark}

\subsection{The category $\Spant$}

Now we introduce a category $\Spant$. Using the notation of subsection \ref{span-quasicategory}, we define $\Spant=\Span(\Arr(\Fin))$, where $\Arr(\Fin)$ is the category of arrows in $\Fin$.

So an $n$-cell of $\Spant$ is a pair of span diagrams $\{X_{ij}\},\{Y_{ij}\}\in\Span_n$ with maps $f_{ij}:X_{ij}\rightarrow Y_{ij}$, which commute with all the structure maps. 

Equivalently, it's a natural transformation between functors $X,Y:C_n\rightarrow \Fin$, where both $X$ and $Y$ have the pullback property.

There's a 2-functor $p:\Spant\rightarrow\Span$ coming from the functor $\Arr(\Fin)\rightarrow\Fin$ which sends $(X\rightarrow Y)$ to $Y$. According to the description above, this sends a morphism of span diagrams to the codomain.

Now, we want to study this functor. First we find a good supply of $p$-cartesian morphisms (as introduced in Definition \ref{defn-cartesian-fibration} above).

\begin{prop}
\label{spant-span-cartesian-morphisms}
Any 1-cell of $\Spant$ of the form
\begin{displaymath}
\xymatrix{Y_0\ar[d]&Y_{01}\ar[l]_{=}\ar[r]\ar[d]\pb{315}&Y_1\ar[d]\\
          X_0      &X_{01}\ar[l]    \ar[r]              &X_1.}
\end{displaymath}
ie. which has the top left map the identity and right-hand square a pullback, is $p$-cartesian.
\end{prop}
\begin{proof}
By Proposition \ref{acyclic-kan}, there are four checks to make on the functor
\[\Spant_{/f}\longrightarrow(\Spant_{/y})\timeso{\Span_{/py}}(\Span_{/pf})\]
to show that it is an acyclic Kan fibration: we must check it has the right lifting property with respect to $\partial\Delta^m\rightarrow\Delta^m$ for $m\leq 3$. We are using the notation $y$ for the 0-cell of $\Spant$ given by $Y_4\rightarrow X_4$.

Firstly, we show the existence of liftings for $\emptyset\rightarrow\Delta^0$.

Given a diagram like the following, which represents a $0$-cell of $\Spant_{/y}\timeso{\Span_{/py}}\Span_{/pf}$,
\begin{displaymath}
\xymatrix{
  &&Y_{24}\ar[dldl]\ar[drdr]\ar[d]&&\\
  &&X_{24}\pb{270}\ar[dl]\ar[dr]&&\\
  Y_2\ar[d]&X_{23}\ar[dl]\ar[dr]&&X_{34}\ar[dl]\ar[dr]&Y_4\ar[d]\\
  X_2&&X_3&&X_4,}
\end{displaymath}
we can fill it in to form a full 0-cell of $\Spant_{/f}$ as follows:
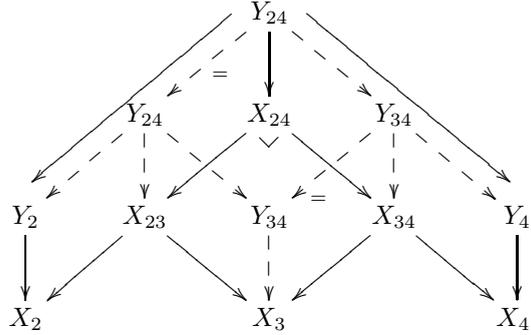
\begin{displaymath}
\xymatrix{
  &&Y_{24}\ar@<-2ex>[dldl]\dar[dl]^{=}\ar@<2ex>[drdr]\dar[dr]\ar[d]&&\\
  &Y_{24}\dar[d]\dar[dl]\dar[dr]&X_{24}\pb{270}\ar[dl]\ar[dr]&Y_{34}\dar[d]\dar[dl]^>>>{=}\dar[dr]&\\
  Y_2\ar[d]&X_{23}\ar[dl]\ar[dr]&Y_{34}\dar[d]&X_{34}\ar[dl]\ar[dr]&Y_4\ar[d]\\
  X_2&&X_3&&X_4,}
\end{displaymath}
and this is the required lifting.

Next, a diagram
\begin{displaymath}
\xymatrix{
\partial\Delta^1\ar[r]\ar[d]&\Spant_{/f}\ar[d]\\
\Delta^1\ar[r]&(\Spant_{/y})\timeso{\Span_{/py}}(\Span_{/pf})}
\end{displaymath}
gives us a configuration of $Y$'s as follows:
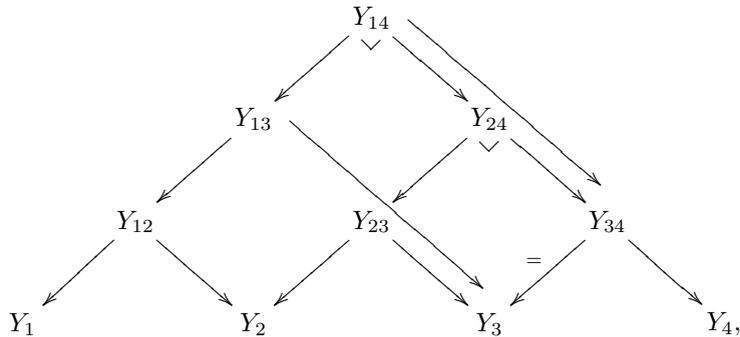
\begin{displaymath}
\xymatrix{
&&&Y_{14}\ar[dl]\ar[dr]\ar@<2ex>[drdr]\pb{270}&&&\\
&&Y_{13}\ar[dl]\ar@<2ex>[drdr]&&Y_{24}\ar[dl]\ar[dr]\pb{270}&&\\
&Y_{12}\ar[dl]\ar[dr]&&Y_{23}\ar[dl]\ar[dr]&&Y_{34}\ar[dl]_{=}\ar[dr]&\\
Y_1&&Y_2&&Y_3&&Y_4,}
\end{displaymath}
where all squares commute and are pullbacks. There is also a full diagram of $X_{ij}$'s, and maps $Y_{ij}\rightarrow X_{ij}$. The parallel morphisms $Y_{14}\rightarrow Y_{34}$ do not have to agree \emph{prima facie}, but the composites $Y_{14}\rightarrow Y_4$ do agree. This maps to a complete span diagram of $X$'s in the obvious way.

However, since $Y_{34}$ is a pullback, the parallel morphisms into it do commute (since the two composites into $Y_4$ and $X_{34}$ do agree).

The maps $Y_{14}\rightarrow Y_{13}$ and $Y_{24}\rightarrow Y_{23}$ are isomorphisms, since they're pullbacks of an isomorphism. This allows us to define a map $Y_{13}\rightarrow Y_{23}$, which makes the resulting top and left squares into pullbacks. Finally, the resulting parallel pair of morphisms $Y_{13}\rightarrow Y_3$ agree, since they are isomorphic to the pair considered earlier.

Now we brace ourselves and consider liftings for $\partial\Delta^2\rightarrow\Delta^2$. Here the morphism $\partial\Delta^2\rightarrow\Spant_{/f}$ gives us a diagram like
\begin{displaymath}
\xymatrix{
&&&&Y_{04}\ar[dl]\ar[dr]\ar@<2ex>[drdr]\pb{270}&&&&\\
&&&Y_{03}\ar[dl]\ar[dr]\ar@<2ex>[drdr]\pb{270}&&Y_{14}\ar[dl]\ar[dr]\pb{270}&&&\\
&&Y_{02}\ar[dl]\ar[dr]\ar@<2ex>[drdr]\pb{270}&&Y_{13}\ar[dl]\ar[dr]\pb{270}&&Y_{24}\ar[dl]\ar[dr]\pb{270}&&\\
&Y_{01}\ar[dl]\ar[dr]&&Y_{12}\ar[dl]\ar[dr]&&Y_{23}\ar[dl]\ar[dr]&&Y_{34}\ar[dl]\ar[dr]&\\
Y_0&&Y_1&&Y_2&&Y_3&&Y_4}
\end{displaymath}
Here all squares are pullbacks, but it is not given that the parallel pairs agree. However, the morphism $\Delta^2\rightarrow\Spant_{/y}$ gives us exactly this necessary extra coherence data, completing this check.

Lastly, it is straightforward to check that, given a lifting problem for $\partial\Delta^3\rightarrow\Delta^3$, all data is given and is coherent: we get a complete span diagram.
\end{proof}

\begin{prop}
\label{span-cart-fib}
The map $p:\Spant\rightarrow\Span$ is a cartesian fibration.
\end{prop}
\begin{proof}
Firstly, we show that the map is an inner fibration. By Proposition \ref{inner-fibs}, we need only check horn extensions for $\Lambda^2_1\rightarrow\Delta^2$.
This gives us the following diagram:
\begin{displaymath}
\xymatrix{
% &&X_{02}\ar@{..>}[dl]\ar@{..>}[dr]\ar@{..>}[d]\pb{270}&&\\
&X_{01}\ar[dl]\ar[dr]\ar[d]&Y_{02}\ar[dl]\ar[dr]\pb{270}&X_{12}\ar[dl]\ar[dr]\ar[d]&\\
X_0\ar[d]&Y_{01}\ar[dl]\ar[dr]&X_1\ar[d]&Y_{12}\ar[dl]\ar[dr]&X_2\ar[d]\\
Y_0&&Y_1&&Y_2.}
\end{displaymath}
This can be filled in to a full map of span diagrams by taking $X_{02}$ to be the pullback of $X_{01}\rightarrow X_1\leftarrow X_{12}$; this maps to $Y_{02}$ in an appropriate manner.

Given a 1-cell $X_3\leftarrow X_{34}\rightarrow X_4$ of $\Span$ (the numbering will make sense later) and an 0-cell $Y_4\rightarrow X_4$ of $\Spant$, we need to find a $p$-cartesian morphism of $\Spant$ which restricts to these two.

But we can define $Y_{34}$ to form a 1-cell of $\Spant$ as follows:
\begin{displaymath}
\xymatrix{Y_{34}\ar[d]&Y_{34}\ar[l]_{=}\ar[r]\ar[d]\pb{315}&Y_4\ar[d]\\
          X_3         &X_{34}\ar[l]    \ar[r]              &X_4.}
\end{displaymath}
This is $p$-cartesian by Proposition \ref{spant-span-cartesian-morphisms} above.
\end{proof}

This construction is compatible with the construction by Lurie \cite{HA}*{Notation 2.4.1.2} of the cartesian fibration $\Gamma^\times\rightarrow\Fin$, in the following sense:
\begin{prop}
There is a commuting diagram
\begin{displaymath}
\xymatrix{\Gamma^\times\ar[r]^{L^\times}\ar[d]&\Spant\ar[d]\\
          \Fins \ar[r]_{L}             &\Span.}
\end{displaymath}
\end{prop}
\begin{proof}
\end{proof}

\subsection{Cartesian morphisms for $\Spant\rightarrow\Span$}

In this section we classify all morphisms which are $p$-cartesian, where $p:\Spant\rightarrow\Span$ is the natural projection map.

For convenience of notation, we will work with the equivalent notion in the opposite categories: classifying $p^\op$-cocartesian morphisms where $p^\op$ is the corresponding morphism ${\Spant}^\op\rightarrow\Span^\op$.

In the proof of Proposition \ref{span-cart-fib}, we showed that a 1-cell $F\in{\Spant}^\op_1$ given by
\begin{displaymath}
  \xymatrix{X_0\ar[d]&X_{01}\ar[l]_{\lambda^X_{01}}\ar[r]^{\rho^X_{01}}\ar[d]&X_1\ar[d]\\
            Y_0      &Y_{01}\ar[l]^{\lambda^Y_{01}}\ar[r]_{\rho^Y_{01}}      &Y_1}
\end{displaymath}
is $p^\op$-cocartesian if the morphism $\lambda^X_{01}$ is an isomorphism, and if the right-hand square is a pullback square.

We write $T_F$ for ${\Spant}^\op_{F_0/}\times_{(\Span^\op_{X_0/})}\Span^\op_{X/}$.

The argument depends on the diagrams used in the proof of Proposition \ref{span-cart-fib}. We will take to drawing the bottom part of a span upside-down: this will simplify the diagrams in practice.

\begin{prop}
\label{cart-pb-surj}
If $F$ is $p^\op$-cocartesian, then the natural map $X_{01}\rightarrow X_0\times_{X_0}X_{01}$ is surjective.
\end{prop}
\begin{proof}
  Given an element $(x,y)\in X_0\times_{Y_0}Y_{01}$, the solid arrows of the following diagram describe a cell $\Delta^0\rightarrow T_F$:
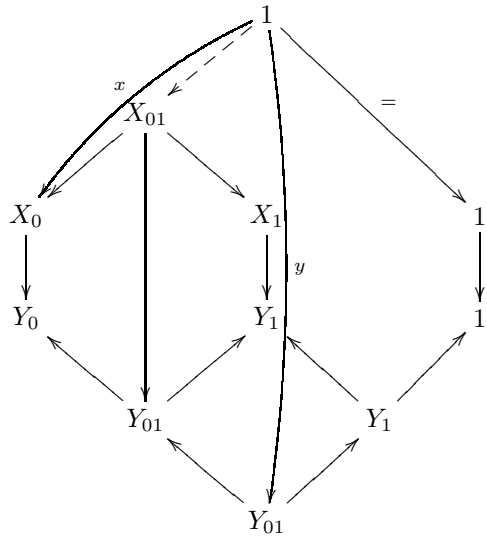
\begin{displaymath}
 \xymatrix{&&1\ar@/_2ex/[ddll]_x\dar[dl]\ar[ddrr]^=\ar@/^1.7ex/[ddddd]^y&&\\
           &X_{01}\ar[ddd]\ar[dl]\ar[dr]&&&\\
           X_0\ar[d]&&X_1\ar[d]&&1\ar[d]\\
           Y_0&&Y_1&&1\\
           &Y_{01}\ar[ul]\ar[ur]&&Y_1\ar[ul]\ar[ur]&\\
           &&Y_{01}\ar[ul]\ar[ur]&&}
\end{displaymath}
We are assuming that an extension to a cell $\Delta^0\rightarrow\Spant_{F/}$ exists; this provides us with the dotted arrow $1\rightarrow X_{01}$: an element of $X_{01}$ which maps to $(x,y)$. This proves surjectivity.
\end{proof}

\begin{prop}
\label{cart-topmap-surj}
If $F$ is $p^\op$-cocartesian, then the map $\rho_{01}^X:X_{01}\rightarrow X_1$ is surjective.
\end{prop}
\begin{proof}
Suppose this is not the case: that $x\in X_1$ has no preimage in $X_{01}$.

We consider a lifting problem for $\partial\Delta^1\rightarrow\Delta^1$ along $\Spant_{/F}\rightarrow T_F$. The data of such a situation is specified by solid arrows of the following diagram:
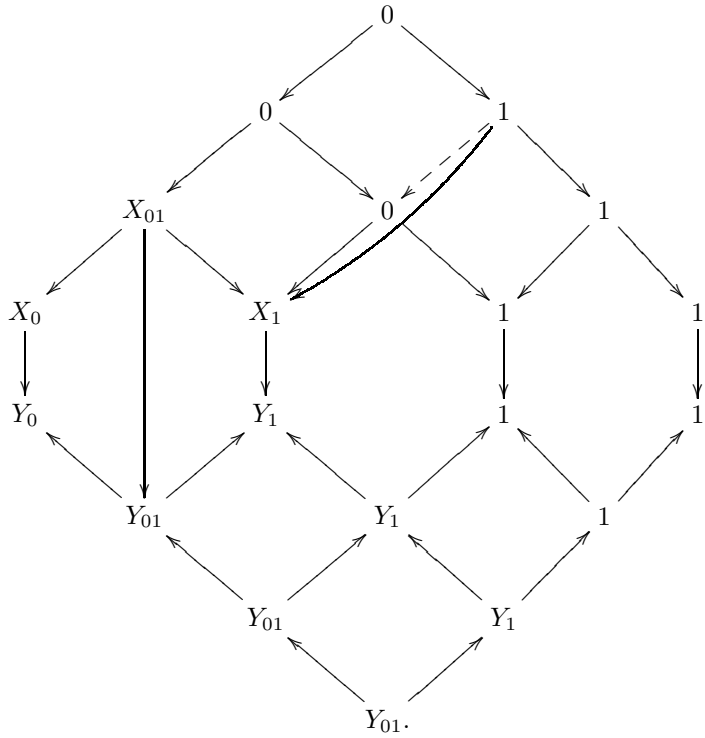
\begin{displaymath}
 \xymatrix{
&&&0\ar[dl]\ar[dr]&&&\\
&&0\ar[dl]\ar[dr]&&1\ar@/^1.7ex/[ddll]\dar[dl]\ar[dr]&&\\
&X_{01}\ar[dl]\ar[dr]\ar[ddd]&&0\ar[dl]\ar[dr]&&1\ar[dl]\ar[dr]&\\
X_0\ar[d]&&X_1\ar[d]&&1\ar[d]&&1\ar[d]\\
Y_0&&Y_1&&1&&1\\
&Y_{01}\ar[ul]\ar[ur]&&Y_1\ar[ul]\ar[ur]&&1\ar[ul]\ar[ur]&\\
&&Y_{01}\ar[ul]\ar[ur]&&Y_1\ar[ul]\ar[ur]&&\\
&&&Y_{01}.\ar[ul]\ar[ur]&&&}
\end{displaymath}
By hypothesis, all the squares in each half are pullbacks.

Since $F$ is assumed to be $p^\op$-cocartesian, an extension exists along the dotted line: a contradiction.
\end{proof}

\begin{prop}
\label{cart-topmap-inj}
If $F$ is $p^\op$-cocartesian, then the map $\rho_{01}^X:X_{01}\rightarrow X_1$ is injective.
\end{prop}
\begin{proof}
Suppose not: that there is $x\in U_1$ with $P={\rho_{01}^X}^{-1}(x)$ a set of size at least 2. Then there is a nontrivial automorphism $\alpha$ of $P$.

We now consider the following lifting problem for $\partial\Delta^1\rightarrow\Delta^1$ along $\Spant_{F/}\rightarrow T_F$, where $i$ is the inclusion $P\rightarrow X_{01}$, and the top parallel collection of morphisms need not commute:
\begin{displaymath} 
\xymatrix{
&&&P\ar[dl]^\alpha\ar@/_1.7ex/[ddll]_i\ar[dr]&&&\\
&&P\ar[dl]^i\ar[dr]&&1\ar@/^1.7ex/[ddll]^(.75){x}\ar[dr]&&\\
&X_{01}\ar[dl]\ar[dr]\ar[ddd]&&1\ar[dl]_x\ar[dr]&&1\ar[dl]\ar[dr]&\\
X_0\ar[d]&&X_1\ar[d]&&1\ar[d]&&1\ar[d]\\
Y_0&&Y_1&&1&&1\\
&Y_{01}\ar[ul]\ar[ur]&&Y_1\ar[ul]\ar[ur]&&1\ar[ul]\ar[ur]&\\
&&Y_{01}\ar[ul]\ar[ur]&&Y_1\ar[ul]\ar[ur]&&\\
&&&Y_{01}.\ar[ul]\ar[ur]&&&}
\end{displaymath}
Again, all squares are pullbacks. By assumption this lifts to a complete diagram $\Delta^1\rightarrow\Spant_{F/}$, meaning that $i\alpha=i$, meaning that $\alpha$ is trivial: a contradiction.
\end{proof}

\begin{prop}
\label{cart-pb-inj}
If $F$ is $p^\op$-cocartesian, then the natural map $X_{01}\rightarrow X_0\times_{Y_0}Y_{01}$ is injective.
\end{prop}
\begin{proof}
Let $(x,y)$ be any element of $X_0\times_{Y_0}Y_{01}$, and let $a,a'$ be two elements of the preimage. We consider another lifting problem for $\partial\Delta^1\rightarrow\Delta^1$ along $\Spant_{F/}\rightarrow T_F$, where again the top parallel collection of morphisms need not commute:
\begin{displaymath} 
\xymatrix{
&&&1\ar[dl]\ar@/_1.7ex/[ddll]_{a'}\ar[dr]&&&\\
&&1\ar[dl]^a\ar[dr]&&1\ar@/^1.7ex/[ddll]\ar[dr]&&\\
&X_{01}\ar[dl]\ar[dr]\ar[ddd]&&1\ar[dl]\ar[dr]&&1\ar[dl]\ar[dr]&\\
X_0\ar[d]&&X_1\ar[d]&&1\ar[d]&&1\ar[d]\\
Y_0&&Y_1&&1&&1\\
&Y_{01}\ar[ul]\ar[ur]&&Y_1\ar[ul]\ar[ur]&&1\ar[ul]\ar[ur]&\\
&&Y_{01}\ar[ul]\ar[ur]&&Y_1\ar[ul]\ar[ur]&&\\
&&&Y_{01}.\ar[ul]\ar[ur]&&&}
\end{displaymath}
The fact that appropriate pullback squares exist follows from Propositions \ref{cart-topmap-surj} and \ref{cart-topmap-inj}. Since $F$ is assumed $p^\op$-cocartesian, the lifting gives us that $a=a'$.
\end{proof}

\begin{thm}
 If $F$ is given by
\begin{displaymath}
  \xymatrix{X_0\ar[d]&X_{01}\ar[l]_{\lambda^X_{01}}\ar[r]^{\rho^X_{01}}\ar[d]&X_1\ar[d]\\
            Y_0      &Y_{01}\ar[l]^{\lambda^Y_{01}}\ar[r]_{\rho^Y_{01}}      &Y_1,}
\end{displaymath}
then it is $p$-cartesian if and only if the right-hand square is a pullback and $\lambda^X_{01}$ is an isomorphism. 
\end{thm}
\begin{proof}
 One direction is Proposition \ref{span-cart-fib}, the other is jointly implied by Propositions \ref{cart-pb-surj}, \ref{cart-topmap-surj}, \ref{cart-topmap-inj}, and \ref{cart-pb-inj}.

 We note that we have not used the lifting condition for $\partial\Delta^2\rightarrow\Delta^2$, and deduce that it is automatically satisfied in the presence of the others: this is apparently not otherwise clear.
\end{proof}

\subsection{Lawvere symmetric monoidal structures}
\label{lawvere-sym-mon-structures}

Given a quasicategory $\cC$ with cartesian products, we shall produce a model of $\Span$ in quasicategories (as defined in subsection \ref{models-in-quasicategories}).

First we define an auxiliary category $\tcCt$. For $K\rightarrow\Span$, we define $\tcCt$ to be the simplicial set represented by the following functor in $K$:
\[\Hom_{\Span}(K,\tcCt)=\Hom(K\timeso{\Span}\Spant,\cC).\]
(It is straightforward to check that this functor does indeed preserve colimits.)

This has the following important structural property: 
\begin{prop}
The projection $\tp:\tcCt\rightarrow\Span$ is a cocartesian fibration.
\end{prop}
\begin{proof}
We have
\[\Hom_\Span(K,\tcCt)=\Hom(K\timeso{\Span}\Spant,\cC)=\Hom_\Span(K\timeso{\Span}\Spant,\cC\times\Span).\]

The map $p:\Spant\rightarrow\Span$ was shown to be a cartesian fibration in Proposition \ref{span-cart-fib}. Since the map $\cC\rightarrow 1$ is evidently a cocartesian fibration, the projection $q:\cC\times\Span\rightarrow\Span$ is also a cocartesian fibration (by \cite{HTT}, 2.3.2.3).

These two maps satisfy the hypotheses for $p$ and $q$ respectively in \cite{HTT}*{Lemma 3.2.2.13}, and so the proposition is proved.
\end{proof}

We can describe the fibre $\tcCt_A$ of $\tcCt$ over a finite set $A\in\Span_0$:
\begin{prop}
\label{fibre-description}
 \[\tcCt_A=\Hom(\Span^A, \cC).\]
\end{prop}
\begin{proof}
 We have:
\begin{align*}
 (\tcCt_A)_n&=\left\{\text{maps}\vcenter{\xymatrix{\Delta^n\ar[rr]\ar[dr]_A&&\tcCt\ar[dl]\\&\Span&}}\right\}\\
            &=\Hom(\Delta^n\timeso{\Span}\Spant,\cC)\\
            &=\Hom(\Delta^n\times\Span(\Fin_{/A}),\cC),
\end{align*}
and so $\tcCt_A$ can be identified with the simplicial set of functors, from the category $\Span(\Fin_{/A})$ of spans of finite sets over $A$, into $\cC$. But since $\Span(\Fin_{/A})=\Span(\Fin^A)=\Span^A$, we get:
\begin{align*}
 \tcCt_A&=\Hom(\Span(\Fin_{/A}), \cC)\\
        &=\Hom(\Span^A,\cC).
\end{align*}
\end{proof}

This description allows us to analyse the $\tp$-cocartesian morphisms in $\tcCt$:
\begin{prop}
  Let $\alpha$ be a morphism in $\tcCt$, with image the 1-cell $X\stackrel{f}{\leftarrow}Z\stackrel{g}{\rightarrow}Y$ in $\Span$. Then $\alpha$ is $\tp$-cocartesian if and only if, for every $U\subset Y$, the morphism $\alpha$ takes the 1-cell
\begin{displaymath}
 \xymatrix{g^*U\ar[d]&g^*U\ar[l]\ar[d]\ar[r]&U\ar[d]\\
           X         &Z   \ar[l]      \ar[r]&Y}
\end{displaymath}
of $\Delta^1\timeso{\Span}\Spant$ to an equivalence in $\cC$.
 
\end{prop}
\begin{proof}
First, we study the morphisms which are cocartesian with respect to the projection functor $q:\cC\times\Span\rightarrow\Span$.

But Proposition \ref{prods-and-carts} makes it clear that these are the products of the morphisms which are cocartesian for $\cC\rightarrow 1$ and those which 
are cocartesian for the identity on $\Span$. By \cite{HTT}*{Remark 2.3.1.4}, the former morphisms are the equivalences in $\cC$, and the latter are all the morphisms in $\Span$.

So the $q$-cocartesian morphisms are those which are an equivalence on the left factor.

With these preliminaries, the result follows by invoking \cite{HTT}*{Lemma 3.2.2.13}.
\end{proof}

Since Proposition \ref{fibre-description} describes the fibre of $\cC\rightarrow\Span$ over an object $A\in\Span_0\isom\Fin_0$ as the category of functors $\Span^A\rightarrow\cC$. We can thus define $\cCt$ to be the full subcategory whose objects are all the product-preserving functors $\Span^A\rightarrow\cC$ for all $A\in\Span_0$.

We define $p:\cCt\rightarrow\Span$ to be the restriction of $\tp$ to $\cCt$, and begin to amass good properties of this functor.
\begin{prop}
\label{cocart-from-prods}
The projection $p:\cCt\rightarrow\Span$ is a cocartesian fibration, with the same cocartesian morphisms as $\tcCt\rightarrow\Span$.
\end{prop}
\begin{proof}
The map $p$ is evidently an inner fibration, since it's a restriction of an inner fibration to a full subcategory.

  So we just need to demonstrate that, given a span $\alpha=(X\stackrel{f}{\leftarrow}Z\stackrel{g}{\rightarrow}Y)$ a 1-cell in $\Span$, and an element $F$ of the fibre $\cCt_Y$, there is cocartesian lift of $\alpha$ in $\tcCt$, whose left-hand vertex $G$ is an element of the fibre $\cCt_X$.

But, in terms of the description of the fibre we are given, if
\[F:\Span(\Fin_{/Y})\rightarrow\cC\]
is product-preserving, then the vertex of the natural lift is given by
\[G(A\rightarrow X)=F(f^*A\rightarrow Y),\]
which is clearly a product-preserving functor $\Span(\Fin_{/X})\rightarrow\cC$.
\end{proof}

\begin{prop}
\label{lifting-prop}
Assume that $\cC$ has finite products. The projection $p:\cCt\rightarrow\Span$ has the property that, given a coproduct diagram in $\Span$, consisting of $A\sqcup B\leftarrow A\rightarrow A$ and $A\sqcup B\leftarrow B\rightarrow B$, the corresponding functors realise $p^{-1}(A\sqcup B)$ as the product of $p^{-1}(A)$ and $p^{-1}(B)$.
\end{prop}
\begin{proof}
The fibre $\cCt_A$ over a finite set $A$ is the quasicategory of product-preserving functors $\Span^A\rightarrow\cC$. If $\cC$ has products, this is isomorphic to $\cC^A$.

It is quick to check that the morphisms given realise the product structures correctly.
\end{proof}

Motivated by the above propositions, we make the following definition (recalling the definition of a cocartesian fibration from \cite{HTT}*{Section 2.4}):
\begin{defn}
A \emph{Lawvere symmetric monoidal structure} is a model of $\Span$ in quasicategories: a cocartesian fibration $\cCt\rightarrow\Span$ such that preimages of product diagrams in $\Span$ are product diagrams of categories.
\end{defn}

So, the propositions above assemble to:
\begin{thm}
An quasicategory with finite products gives a Lawvere symmetric monoidal structure.
\end{thm}
\begin{proof}
This is Propositions \ref{cocart-from-prods} and \ref{lifting-prop}.
\end{proof}

We need notions of functor too:
\begin{defn}
A \emph{symmetric monoidal functor} between two Lawvere symmetric monoidal structures $\cCt\stackrel{p}{\rightarrow}\Span$ and $\cDt\stackrel{q}{\rightarrow}\Span$ is a map of models of $\Span$ in quasicategories: that is, a functor taking $p$-cocartesian morphisms to $q$-cocartesian morphisms.

A \emph{lax symmetric monoidal functor} is one which takes $p$-cocartesian morphisms whose image in $\Span$ is collapsing to $q$-cocartesian morphisms.
\end{defn}

\subsection{Lawvere commutative algebra objects}

Now we seek to define an algebra object in a Lawvere symmetric monoidal category.

We say that an 1-cell in the category $\Span$ is \emph{collapsing} if it is of the form
\[\xymatrix{X&Z\iar[l]\ar[r]^\sim&Y},\]
or equivalently if it's isomorphic to a diagram of the form
\[\xymatrix{A\sqcup B&A\iar[l]\ar[r]^=&A}.\]
We can thus define a subquasicategory $\Spanc$ of $\Span$ containing all objects, all collapsing 1-cells and all higher cells all of whose edges are collapsing.

We define a morphism in $\Fins$ to be collapsing if all the preimage of every element except the basepoint has size exactly one. By the same process we can define a subquasicategory $\Finsc$, and it is quick to check we have a pullback diagram
\[\xymatrix{\Finsc\pb{315}\ar[d]\ar[r]&\Spanc\ar[d]\\\Fins\ar[r]_R&\Span.}\]
In other words: a morphism in $\Fins$ is collapsing if and only if it yields a collapsing span diagram.

We define an algebra object in a style analogous to those found in \cite{HA}: a \emph{Lawvere commutative algebra object} of a Lawvere symmetric monoidal category $\cCt\stackrel{p}{\longrightarrow}\Span$ consists of a section $f$ of $p$ which takes collapsing morphisms to $p$-cocartesian morphisms.

We can then define the quasicategory of Lawvere commutative algebra objects $\Algt(\cCt)$ to be the full subquasicategory of $\Fun(\Span,\cCt)$ on the Lawvere commutative algebra objects. When the monoidal structure is understood, which is most of the time, we write $\Algt(\cC)$ instead.

\section{Comparing approaches to $E_\infty$-spaces}
\label{comparison-theorems}

The purpose of this section is to show that Lurie's definition of a symmetric monoidal category in \cite{HA} is equivalent to the one advanced in the preceding section.

We begin with some reasonably lightweight sections, which sketch straightforward simple arguments why $\Span$ should be thought of as the theory for commutative monoids in the setting of quasicategories.

The purpose of this is to compare our theory $\Span$ (both as a quasicategory, and via its models) with more traditional ways of defining $E_\infty$ structures. Our comparison, ultimately, shall be with the Barratt-Eccles operad.

\subsection{Operads and Theories}

Now, we start by proving a basic (and standard) result giving the structure of free operadic algebras.

We state a standard proposition:
\begin{prop}
Let $\cO$ be an operad in simplicial sets, and $X$ be a finite set. The free $\cO$-algebra on $X$, denoted $\cO[X]$ has underlying space given by
\[|\cO[X]| = \coprod_{n\geq0}(\cO(n)\times_{\Sigma_n}X^n) = \colim_{A\in\FSI}(\cO(A)\times X^A).\]

Thus $|\cO[X]|^B$ is given by
\[|\cO[X]|^B = \colim_{(A\stackrel{f}{\rightarrow}B)\in(\Fin_{/B})^{\isom}}(\cO(f)\times X^A),\]
where the colimit is taken in the category whose objects are finite sets with maps to $B$, and whose morphisms are isomorphisms over $B$.

And, in this notation, the $\cO$-algebra structure map $\cO(B)\times|\cO[X]|^B\rightarrow|\cO[X]|$ is simply given by the operad composition $\cO(B)\times\cO(f)\rightarrow\cO(A)$.
\end{prop}
\begin{comment}
\begin{proof}
We seek to show that, using this definition, for any $\cO$-algebra $A$ we have
\[\Alg_{\cO}(\cO[X],A) \isom \Spaces(X,|A|).\]
There is a map $i:X\rightarrow |\cO[X]|$ taking $X\rightarrow\cO(1)\times X\isom C$; composing this with $|\cO[X]|\rightarrow|A|$ gives us our map $\Alg_{\cO}(\cO[X],A) \rightarrow \Spaces(X,|A|)$.

The composite
\[\cO(A)\times X^A\stackrel{1\times i^A}{\longrightarrow}\cO(A)\times |\cO[X]|^A\longrightarrow|\cO[X]|,\]
where the second map is the operadic action, agrees with the structure maps $\cO(A)\times X^A\rightarrow |\cO[X]|$.

Thus a map $X\rightarrow |A|$ completely determines 
\end{proof}
\end{comment}

So, a point of $\cO[X]$ consists of a finite set $U$ over $X$, and an element of $\cO(U)$.

Moreover, by the free-forgetful adjunction proved in the Proposition above, a map $\cO[Y]\rightarrow\cO[X]$ is the same as a map of spaces $Y\rightarrow|\cO[X]|$, and so is equivalent to a span diagram
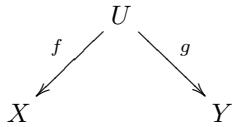
\begin{displaymath}
\xymatrix{&U\ar[dl]_f\ar[dr]^g&\\
X&&Y}
\end{displaymath}
together with a point of $\cO(g)$. The correspondence works as follows: taking preimages of elements in $Y$ give a $Y$-indexed family of sets $U_y$, equipped with maps $U_y\rightarrow X$ and elements of $\cO(U_y)$.

More is true: an $n$-simplex $\alpha\in\Map(\cO[Y],\cO[X])_n$ is equivalent to a span diagram
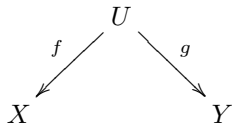
\begin{displaymath}
\xymatrix{&U\ar[dl]_f\ar[dr]^g&\\
X&&Y}
\end{displaymath}
together with an element of $\cO(g)_n$. Indeed, the various contributions from different span diagrams are disconnected from one another.

This suggests a deep connection between span diagrams and operads, which we will exploit in one situation very shortly.

Now, given an operad $\cO$ in simplicial sets, we define a simplicial category $\ThS{\cO}$:
\begin{itemize}
\item Objects of $\ThS{\cO}$ are finite sets.
\item The homspace $\ThS{\cO}(X,Y)$ is the mapping space $\Alg_{\cO}(\cO[Y],\cO[X])$, the full sub-Kan-complex of $\Spaces(|\cO[Y]|,|\cO[X]|)$ whose $0$-simplices are maps preserving operadic structure.
\end{itemize}

Hence (by the full/forgetful adjunction) we also have
\[\ThS{\cO}(X,Y) = \Map(Y,|\cO[X]|),\]
and so in particular
\[\ThS{\cO}(X,Y)_n = \sSet(Y\times\Delta^n,|\cO[X]|)\]

Later we will write $\Th{\cO}$ for the quasicategory $\fC(\ThS{\cO})$. But in the meanwhile, the following explains the relevance of the definition.

\begin{prop}
Functors $\ThS{\cO}\rightarrow\Spaces$ that are product-preserving (in the strict sense that they send coproducts of sets to category-theoretic products of simplicial sets) are equivalent to $\cO$-algebras.
\end{prop}
\begin{proof}
Firstly, given a $\cO$-algebra $A$, we produce a functor $F_A$.

We define $F_A(X)$ to be $A^X$ for any finite set $X$. Now, suppose given finite sets $X$ and $Y$ and an $n$-simplex $\alpha\in\ThS{\cO}(X,Y)_n$, which is the same as a map $Y\times\Delta^n\rightarrow|\cO[X]|$. We must define an induced map $F_A^*(\alpha)\in\Map(A^X,A^Y)_n$.

By restricting to coordinates, it suffices to take a map $\alpha:\Delta^n\rightarrow|\cO[X]|$ and get an induced map $F_A^*(\alpha)\in\Map(A^X,A)_n$; or, more straightforwardly, to define a map $|\cO[X]|\times A^X\rightarrow A$.

By our construction of $|\cO[X]|$ as a disjoint union above, for each finite set $Z$ we need a $Z$-equivariant map
\[\cO(Z)\times X^Z\times A^X\longrightarrow A.\]
However, such a map exists: we define it by first using the composition map $X^Z\times A^X\rightarrow A^Z$, and then using the $\cO$-algebra structure on $A$.

Equivalently, regarding a map $\cO[Y]\rightarrow\cO[X]$ as a span diagram 
\begin{displaymath}
\xymatrix{&U\ar[dl]_f\ar[dr]^g&\\
X&&Y,}
\end{displaymath}
together with an element of $\cO(g)$, we find the corresponding map $A^X\rightarrow A^Y$ is the composition of the pullback map $f^*:A^X\rightarrow A^U$ and an application of the operadic action map $\cO(g)\times A^U\rightarrow A^Y$.

It is necessary to check that this assignment is indeed functorial: if we have  finite sets $X$, $Y$ and $Z$, we must check that the diagram
\begin{displaymath}
\xymatrix{\Map(\cO[Z],\cO[Y])\times\Map(\cO[Y],\cO[X])\ar[r]\ar[d]&\Map(\cO[Z],\cO[X])\ar[d]\\
\Map(A^Y,A^Z)\times\Map(A^X,A^Y)\ar[r]&\Map(A^X,A^Z)}
\end{displaymath}
commutes.

For 0-cells corresponding to span diagrams
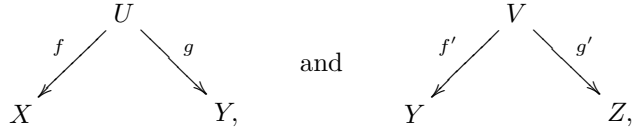
\begin{displaymath}
\vcenter{\xymatrix{&U\ar[dl]_f\ar[dr]^g&\\
X&&Y,}}
\qquad\text{and}\qquad
\vcenter{\xymatrix{&V\ar[dl]_{f'}\ar[dr]^{g'}&\\
Y&&Z,}}
\end{displaymath}
it's immediate to check that both composites correspond to the composite span
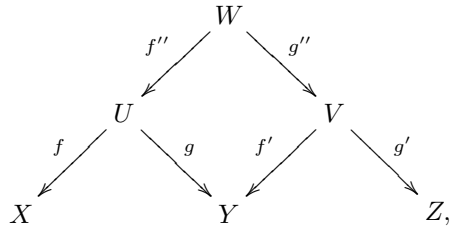
\begin{displaymath}
\xymatrix{&&W\ar[dl]_{f''}\ar[dr]^{g''}&&\\
&U\ar[dl]_f\ar[dr]^g&&V\ar[dl]_{f'}\ar[dr]^{g'}&\\
X&&Y&&Z,}
\end{displaymath}
with given element of $\cO(g'\circ g'')$ obtained by pulling back the element of $\cO(g)$ to obtain an element of $\cO(g'')$ and then using the operadic compition with the given element of $\cO(g')$.

This argument works similarly for higher cells, using the description above of them as span diagrams labelled by higher cells of a product of parts of the operad. This completes this part.

Now we shall go the other way, showing that any functor $F:\ThS{\cO}\rightarrow\Spaces$ has an underlying $\cO$-action on the value of $F$ on the singleton, $F(*)$.

But we get the operadic action from the span diagrams
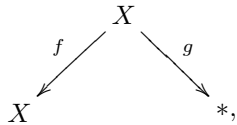
\begin{displaymath}
\xymatrix{&X\ar[dl]_f\ar[dr]^g&\\
X&&\ast,}
\end{displaymath}
as these describe exactly maps $\cO(X)\times A^X\rightarrow A$.

We get compatibility with the operadic product, for a family $Y\rightarrow X$, by considering the composite span diagram
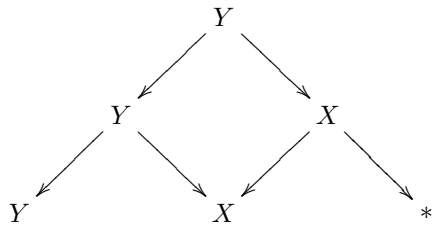
\begin{displaymath}
\xymatrix{&&Y\ar[dl]\ar[dr]&&\\
&Y\ar[dl]\ar[dr]&&X\ar[dl]\ar[dr]&\\
Y&&X&&\ast}
\end{displaymath}

It is quick to check that these two constructions are mutually inverse.
\end{proof}

\subsection{$\Span$ and the Barratt-Eccles operad}
\label{homs-in-span}

Another small piece of propaganda is provided by calculating the homspaces in the quasicategory $\Span$; this makes plausible much of the relationship between $\Span$ and various recognisably classical notions such as the Barratt-Eccles operad \cite{BarEcc}.

One should think of $\Span$ as a homotopy-theoretic elaboration of the theory of commutative monoids. Indeed, the theory of commutative monoids is the opposite of the full subcategory of $\Monoids$ on the objects $\N^r$, and there is a natural forgetful morphism $\Span\rightarrow\Th{\Monoids}$.

An object $X\in\Span_0$ is sent to the object $\N^X$, and a span $X\stackrel{f}{\leftarrow} Z\stackrel{g}{\rightarrow} Y$ is sent to the map
\begin{align*}
\N^X&\longleftarrow\N^Y\\
y\in Y&\longmapsto \sum_{g(z)=y}f(z).
\end{align*}

Thus the theory $\Span$ consists of finitely-generated free commutative monoids, with some unusual autoequivalences on the morphisms.

Writing $\FSI$ for the category of finite sets and isomorphisms, we have
\begin{align*}
\Span(1,1)&=N(\FSI)\\
          &\isom\coprod_{X\in\operatorname{skeleton}(\Fin)}B\Aut(X)\\
          &\isom\coprod_n B\Sigma_n
\end{align*}
(where the sum on the second line is taken over a set of representatives of the isomorphism classes)
This is a famous model for the free $E_\infty$-monoid on one generator: it's that provided by the Barratt-Eccles operad.

The composition of two spans $1\leftarrow X\rightarrow 1$ and $1\leftarrow Y\rightarrow 1$ is given by $1\leftarrow X\times Y\rightarrow 1$, and on hom-spaces is given by the maps $B\Sigma_n\times B\Sigma_{n'}\rightarrow B\Sigma_{n\times n'}$ induced by the Cartesian product map $\Sigma_n\times\Sigma_{n'}\rightarrow\Sigma_{n\times n'}$.

Moreover, we also have
\begin{align*}
\Span(X,Y)&=N(\FSI_{/X,Y})\\
          &=N(\FSI_{/X\times Y})\\
          &=\left(\coprod_n B\Sigma_n\right)^{X\times Y},
\end{align*}
so we can think of the space of spans from $X$ to $Y$ as being $X$-by-$Y$ matrices with entries in the free $E_\infty$-monoid on one generator. Composition is then matrix multiplication, and it is easily seen that coproducts and products are given on homspaces by block sums of matrices.

Let $\cE$ be the Barratt-Eccles operad. This has $\cE(X) = E(\Sigma_X)$, and operadic composition induced by the map of sets
\[\Sigma_X\times\prod_{x\in X}\Sigma_{Y_x}\longrightarrow\Sigma_Y.\]

We would like to be able to compare the weak $2$-categories $\Span2$ and the groupoid-enriched category $\Th{\cE}^\Delta$. But this requires more machinery than we wish to develop. So we do some hard work, and compare $\Span$ and $\Th{\cE}$ themselves.

\begin{prop}
There is a natural equivalence $\Psi:\Th{\cE}\rightarrow\Span$.
\end{prop}
\begin{proof}
Since $\Span$ is a $2$-category, we must only construct $\Psi$ for $2$-simplices and then show that $3$-simplices can be sent to the identity.

Both categories have finite sets as objects, and so we take $\Psi$ to be the identity on objects.

Now, an element of $(\Th{\cE})_1$ consists of a morphism $\fC[1]\rightarrow\Th{\cE}^\Delta$, in other words two finite sets $Y$ and $X$, and a point of $\Th{\cE}(Y,X)$, a point of 
\[\colim_{f:A\rightarrow Y}\cE(f)\times X^A.\]
By forgetting the $\cE(f)$ (which has only a single $0$-simplex anyway), we get a span diagram $X\leftarrow A\rightarrow Y$.

An element $(\Th{\cE})_2$ gives us three such span diagrams by restricting to the three natural inclusions $\fC[1]\rightarrow\fC[2]$, agreeing on objects thus:
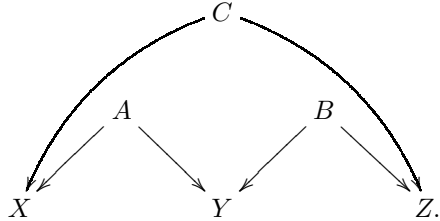
\begin{displaymath}
\xymatrix{
  &&C\ar@/_3ex/[ddll]\ar@/^3ex/[ddrr]&&\\
&A\ar[dl]\ar[dr]&&B\ar[dl]\ar[dr]&\\
X&&Y&&Z.}
\end{displaymath}
We also obtain a homotopy between the composite of the two short spans and the long one; this gives an isomorphism from $C$ to some model of the pullback $A\times_YB$ and thus provides a $2$-cell of $\Span$.

From a $3$-cell of $\Th{\cE}$ we get four $2$-cells of $\Span$, which assemble to form a diagram as follows,
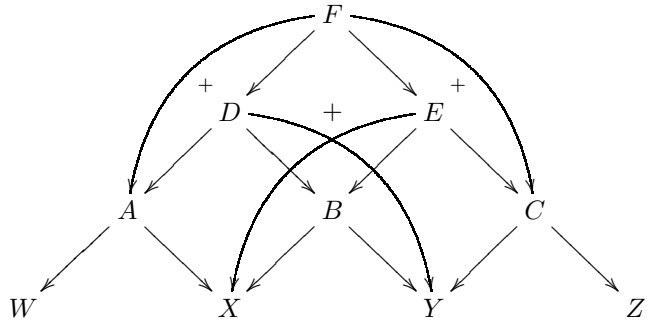
\begin{displaymath}
\xymatrix{
&&&F\ar[dl]\ar[dr]\ar@/_5ex/[ddll]^{+}\ar@/^5ex/[ddrr]_{+}&&&\\
&&D\ar[dl]\ar[dr]\ar@/^5ex/[ddrr]&{+}&E\ar[dl]\ar[dr]\ar@/_5ex/[ddll]&&\\
&A\ar[dl]\ar[dr]&&B\ar[dl]\ar[dr]&&C\ar[dl]\ar[dr]&\\
W&&X&&Y&&Z}
\end{displaymath}
where everything is known to commute except the punctured polygons.

However, the map $\fC[3]\rightarrow\Th{\cE}^\Delta$ also provides a map $(\Delta^1)^2\rightarrow\Th{\cE}(W,Z)$, giving a commuting square
\begin{displaymath}
\xymatrix{F\ar[r]\ar[d]&D\timeso{Y}C\ar[d]\\
A\timeso{X}E\ar[r]&A\timeso{X}B\timeso{Y}C,}
\end{displaymath}
which makes the diagram commute as required.

Now we must go on to show that it is an equivalence. But this is not difficult: we shall show (using the philosophy of \cite{HTT}*{Section 1.1.3}) that $\Psi$ is an equivalence since it's the identity on objects and induces weak equivalences on homspaces.

Our calculation in Subsection \ref{homs-in-span} above, and the definition of $\Th{\cE}$ show that both have homspaces
\[\Span(X,Y)\isom\Th{\cE}(X,Y)\isom N(\Fin^{\isom}_{/(X\times Y)}.\]
These homspaces are classifying spaces of groupoids; thus the definition of $\Psi$ gives a weak equivalence immediately.
\end{proof}

\subsection{Discrete Lawvere commutative monoids}

Lurie shows \cite{HTT}*{1.2.3.1} that the nerve functor (of ordinary categories) from ordinary categories to quasicategories has a right adjoint, denoted $\Ho$, the \emph{homotopy category}. We use this theory briefly to understand how $\Span$ really is a quasicategorical version of the algebraic theory of commutative monoids.

We can compute the homotopy category of $\Span$:

\begin{prop}
\label{ho-span}
 The homotopy category $\Ho(\Span)$ of $\Span$ has finite sets as objects and isomorphism classes of span diagrams as morphisms. Composition is by pullback.
\end{prop}
\begin{proof}
 The only check is that the composition in $\Span$ respects isomorphism classes; this is evident.
\end{proof}

This category is recognisable as the algebraic theory for discrete commutative monoids.

This allows us to state the following:
\begin{thm}
\label{discrete-monoids}
Discrete Lawvere commutative monoid objects are the same as commutative monoids.
\end{thm}
\begin{proof}
Since, as mentioned above, $\Ho$ is the right adjoint of the nerve functor $\Cat\rightarrow\sSet$, for any quasicategory $\cC$ all product-preserving functors $\Span\rightarrow\cC$ factor uniquely through the product-preserving functor $\Span\rightarrow\Ho(\Span)$.
\end{proof}

This will be evident from the structure results proven later in this section, but this elementary argument is nevertheless informative, in our opinion.

\subsection{General comparisons}

Our construction has strictly more data than Lurie's construction:
\begin{thm}
Any Lawvere monoidal quasicategory has an underlying symmetric monoidal quasicategory $\cCot\rightarrow\Fins$ (as defined in \cite{HA}*{Chapter 2}).
\end{thm}
\begin{proof}
We define $\cCot$ to be the pullback 
\begin{displaymath}
 \xymatrix{\cCot\ar[r]\ar[d]&\cCt\ar[d]^p\\
           \Fins\ar[r]_L    &\Span}
\end{displaymath}
The left-hand arrow is a cartesian fibration, by \cite{HTT}*{Lemma 2.4.2.3}. The required product property is immediate, since unions in $\Fins$ get sent to product diagrams in $\Span$.
\end{proof}

We note that, by further restriction of structure, we can think of $p^{-1}(*)$ as the \emph{underlying quasicategory} of $\cCt$.

Moreover, maps of Lawvere symmetric monoidal quasicategories yield maps of their underlying Lurie symmetric monoidal quasicategories.

Similarly, commutative algebra objects for $\cCt$ yield commutative algebra objects in $\cCot$: given a section $a$ of $\cCt\rightarrow\Span$, we can pull back to obtain a section of $\cCot\rightarrow\Fins$:
\begin{displaymath}
 \xymatrix{\Fins\ar[r]\ar[d]\pb{315}&\Span\ar[d]\\
           \cCot\ar[r]\ar[d]\pb{315}&\cCt \ar[d]\\
           \Fins\ar[r]              &\Span.}
\end{displaymath}
This defines a functor $\theta:\Algt(\cCt)\rightarrow\Algot(\cCot)$.

We repeat Definition \ref{lawvere-monoid-object}: we define a \emph{Lawvere monoid object} in a category $\cC$ to be a model of $\Span$ in $\cC$: that is, a product-preserving functor $\Span\rightarrow\cC$.

Similarly, for the present argument we define a \emph{Lurie monoid object} in $\cC$ to be a functor $\Fins\rightarrow\cC$ which takes the projection maps of disjoint unions to product diagrams in $\cC$. (This agrees with the definition in \cite{HA}*{Remark 2.4.2.2}).

We can form quasicategories $\Mont(\cC)$ and $\Monot(\cC)$ of Lawvere and Lurie monoid objects respectively, as full subquasicategories of the functor categories $\Fun(\Span,\cC)$ and $\Fun(\Fins,\cC)$ on the monoid objects.

The good properties of the functor $L:\Fins\rightarrow\Span$ defines a functor $\varphi:\Mont(\cC)\rightarrow\Monot(\cC)$, defined by precomposition.

Now we have our two major comparison results:
\begin{thm}
 \label{lawvere-lurie-mons}
The natural functor $\varphi:\Mont(\cC)\rightarrow\Monot(\cC)$, as defined above, is an equivalence.
\end{thm}

\begin{thm}
 \label{lawvere-lurie-algs}
Let $q:\cCt\rightarrow\Span$ be a Lawvere monoidal category, and $\cCot\rightarrow\Fins$ be the underlying Lurie monoidal category.

Then the natural functor $\theta:\Algt(\cC)\rightarrow\Algot(\cC)$ between the corresponding quasicategories of algebras, as defined above, is an equivalence.
\end{thm}

Both will be proved in the next section; here is the most important corollary:
\begin{thm}
\label{lawvere-lurie-cats}
 The quasicategory of Lawvere symmetric monoidal categories is equivalent to the quasicategory of Lurie symmetric monoidal categories.
\end{thm}

\begin{proof}
 The quasicategory of cocartesian fibrations over $\Fins$ is equivalent to the quasicategory of functors $\Fins\rightarrow\Cinfty$; and the disjoint union property for cocartesian fibrations is equivalent to taking disjoint unions to products.

So the category of Lurie symmetric monoidal categories is equivalent to $\Monot(\Cinfty)$.

Similarly, the quasicategory of cocartesian fibrations over $\Span$ is equivalent to the category of functors $\Span\rightarrow\Cinfty$; and the product property is equivalent to being product-preserving.

So the category of Lawvere symmetric monoidal categories is equivalent to $\Mont(\Cinfty)$.

Theorem \ref{lawvere-lurie-mons} gives the required equivalence to prove the theorem.
\end{proof}

\subsection{The proofs of the comparison theorems}

This section merely contains the proofs of the two basic comparison results stated in the last section. The two arguments are very similar.

Both employ an auxiliary quasicategory $\cJ$, defined as follows:
\[\cJ_n=\coprod_{\substack{i+j+1=n,\\i,j\geq -1}}\left\{\text{$f:\Delta^j\rightarrow\Fins$, $g:\Delta^i\star\Delta^j\rightarrow\Span$, with $g|_{\Delta^j}=L\circ f$.}\right\},\]
where the face and degeneracy maps are obvious.

We can draw cells of $\cJ$ as span diagrams equipped with a fenced-off sub-span diagram on the right-hand side, where the fenced-off part is in the essential image of $L$:
\begin{displaymath}
\xymatrix@!R=6mm@!C=6mm
{&&&&X_{04}\pb{270}\ar[dr]\ar[dl]&&&&\\
&&&X_{03}\pb{270}\ar[dr]\ar[dl]&&X_{14}\pb{270}\ar[dr]\ar[dl]&\dl[dddlll]&&\\
&&X_{02}\pb{270}\ar[dr]\ar[dl]&&X_{13}\pb{270}\ar[dr]\ar[dl]&&X_{24}\pb{270}\ar[dr]\iar[dl]&&\\
&X_{01}\ar[dr]\ar[dl]&&X_{12}\ar[dr]\ar[dl]&&X_{23}\ar[dr]\iar[dl]&&X_{34}\ar[dr]\iar[dl]&\\
X_0&&X_1&&X_2&&X_3&&X_4}
\end{displaymath}

When $n=0$, we have $\Ob\cJ=\Ob(\Span)\sqcup\Ob(\Fins)$ (the summands are the contributions of $i=0, j=-1$ and $i=-1, j=0$ respectively).

The full subcategory $\cJ_\Span$ on the objects corresponding to $\Ob(\Span)$ consists of the contributions by $i=n, j=-1$, and this is a copy of $\Span$; the full subcategory $\cJ_\Fins$ on the objects corresponding to $\Ob(\Fins)$ consists of the contributions by $i=-1, j=n$, and this is a copy of $\Fins$.

Note that there are no 1-cells from $\cJ_\Fins$ to $\cJ_\Span$ in $\cJ$. Note also that the full subcategory embedding $\Span\rightarrow\cJ$ admits a retraction $\cL:\cJ\rightarrow\Span$ (defined by $L$ on $\cJ_\Fins$).

Equipped with this, we can prove the results:

\begin{proof}[of Theorem \ref{lawvere-lurie-mons}]
We write $\bMon(\cC)$ for the full subcategory of functors $\Map(\cJ,\cC)$ on the objects $f$ such that:
\begin{enumerate}[(i)]
 \item\label{m-equivs} For any set $A$, the canonical 1-cell of $\cJ$ defined by the constant maps $A_+:\Delta^0\rightarrow\Fins$ and $A:\Delta^0\star\Delta^0\rightarrow\Span$ is sent by $f$ to an equivalence in $\cC$,
 \item\label{m-lawvere} The restriction $f|_{\cJ_\Span}$ is a Lawvere monoid object.
 \item\label{m-lurie} The restriction $f|_{\cJ_\Fins}$ is a Lurie monoid object
\end{enumerate}

We notice that, in the presence of condition (\ref{m-equivs}), conditions (\ref{m-lawvere}) and (\ref{m-lurie}) are equivalent to one another. Indeed, both the sets of collapsing morphisms and the product properties match up under the given equivalences.

Also, condition (\ref{m-equivs}) is equivalent to saying that $f$ is a left Kan extension of $f|_{\Span}$ along $\Span\rightarrow\cJ$ (as defined in \cite{HTT}*{Section 4.3.2}).

Indeed, to say that the following diagram is a Kan extension diagram
\begin{displaymath}
 \xymatrix{\cJ_\Span\ar[r]^{f|_{\cJ_\Span}}\ar[d]&\cC\\
           \cJ\ar[ur]^f&}
\end{displaymath}
is to say that, for every object $X\in\Ob(\cJ)$, the diagram
\begin{displaymath}
 \xymatrix{(\cJ_\Span)_{/X}\ar[r]^{f}\ar[d]&\cC\\
           (\cJ_\Span)_{/X}\star 1\ar[ur]^f&}
\end{displaymath}
makes $f(X)$ a colimit of $f_{/X}$. Here $(\cJ_\Span)_{/X}$ is notation for $(\cJ_{/X})\timeso{\cJ}(\cJ_{\Span})$.

If $X\in\Ob(\cJ_\Span)$, then $(\cJ_\Span)_{/X}\isom\Span_{/X}$, and the diagram is vacuously a colimit.

If $X_+\in\Ob(\cJ_\Fins)$, then $(\cJ_\Span)_{/X}\isom\Span_{/X}$, and the diagram is a colimit if and only if the 1-cell from $A$ to $A_+$ is taken to an equivalence in $\cCt$.

We can show that every map $f_0:\Span\rightarrow\cC$ admits $f_0\circ\cL$ as a left Kan extension to a map $\cJ\rightarrow\cC$.
Indeed, following the definition of a Kan extension along an inclusion, this amounts to showing for $X\in(\Fins)_0$ that the diagram
\begin{displaymath}
\xymatrix{\Span_{/X}\ar[d]\ar[rr]&&\Span\ar[r]^{f_0}&\cC\\
1\star\Span_{/X}\ar[r]&\cJ\ar[r]_{\cL}&\Span\ar[ur]_{f_0}&}
\end{displaymath}
is a colimit diagram of shape $\Span_{/X}$ in $\cC$. However, $\Span_{/X}$ has a terminal object given by the diagram of identities:
\begin{displaymath}
\xymatrix{&X\ar[dl]\ar[dr]&\dl[dl]\\
X&&X}
\end{displaymath}
Thus the colimit is given by $f_0(X)$, with colimiting structure maps described by $\cL$.

Hence we can use \cite{HTT}*{Prop 4.3.2.15} to deduce that the restriction functor $p:\bMon(\cC)\rightarrow\Mont(\cC)$ is acyclic Kan. 

Now, composition with $\cL$ defines a section of the functor $\bMon(\cC)\rightarrow\Mont(\cC)$, and $\theta$ is the composition of this with the restriction map $p':\bMon(\cC)\rightarrow\Monot(\cC)$.

Thus all we need to do is show that $p'$ is acyclic Kan, and \cite{HTT}*{Prop 4.3.2.15} says that this will follow from these two claims:
\begin{enumerate}[(a)]
\item\label{mak-a} Every $f_0\in\Monot(\cC)$ admits a right Kan extension $f$, as shown fitting into the following diagram:
\begin{displaymath}
  \xymatrix{\Fins\ar[r]^{f_0}\ar[d]&\cC\\
            \cJ\ar[ur]_f&}
\end{displaymath}
\item\label{mak-b} Given $f\in\sSet_\Span(\cJ,\cC)$ such that $f_0=f|_{\Fins}$ is a Lurie monoid object, $f$ is a right Kan extension of $f_0$ if and only if $f$ satisfies condition (\ref{m-equivs}) above.
\end{enumerate}

To prove (\ref{mak-a}), for any object $K\in\Span$, we consider the quasicategory \[\cK_K=\Fins\timeso{\Span}(\Span_{K/}).\]
We write $g_K$ for the composite $\cK_K\rightarrow\Fins\rightarrow\cC$.

According to \cite{HTT}*{Lemma 4.3.2.13}, it will suffice to to show that, for every $K$, $g_K$ has a colimit in $\cC$.

Since there is an injection $\Fins\rightarrow\Span$, there is an injection $\cK_K\rightarrow\Span_{K/}$; we thus write the objects of $\cK_K$ as morphisms $K\stackrel{a}{\leftarrow}Y\stackrel{b}{\rightarrow}Z$ of $\Span$. We let $\cK'_K$ denote the full subcategory on the objects where $b$ is an isomorphism and $a$ an injection.

The inclusion $\cK'_K\rightarrow\cK_K$ has a right adjoint. Indeed, one choice of right adjoint sends the object $K\stackrel{a}{\leftarrow}Y\stackrel{b}{\rightarrow}Z$ to $K\leftarrow \im(a)\rightarrow \im(a)$. Regarding an adjunction as a bicartesian fibration over $\Delta^1$, we need to provide cartesian lifts of the nontrivial 1-cell in $\Delta^1$; a lift for $K\stackrel{a}{\leftarrow} Y\stackrel{b}{\rightarrow}Z$ is given by
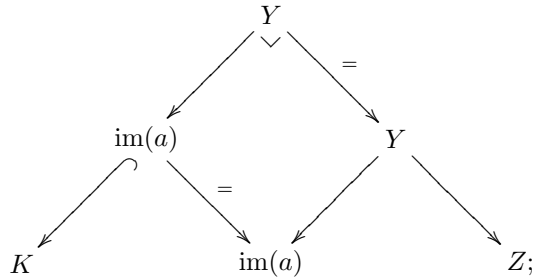
\begin{displaymath}
\xymatrix@!R=8mm@!C=8mm{
&&Y\ar[dl]\ar[dr]^=\pb{270}&&\\
&\im(a)\iar[dl]\ar[dr]^=&&Y\ar[dl]\ar[dr]&\\
K&&\im(a)&&Z;}
\end{displaymath}
it is readily checked that this is indeed cartesian.

Hence $(\cK'_K)^\op\rightarrow(\cK_K)^\op$ is cofinal, so we just need to show that $g'_K=g_K|_{\cK'_K}$ has a colimit in $\cC$.

We write $\cK''_K$ for the full subcategory of $\cK'_K$ on the objects $K\leftarrow\{k\}\rightarrow\{k\}$. Because of the product property of monoidal objects, $g''_K=g'_K|_{\cK''_K}$ is a Kan extension of $g'_K$.

Thus, using \cite{HTT}*{Lemma 4.3.2.7}, we merely need to show that $g''_K$ has a limit in $\cC$. But by the product property, $f$ exhibits $f(K)$ as a limit of $g''_K$, thus proving~(\ref{mak-a}).

The argument is reversible: it shows that $f$ is a right Kan extension of $f_0$ at $K_+$ if and only if $f$ induces an equivalence $f(K_+)\rightarrow f(K)$; this proves~(\ref{mak-b}).
\end{proof}

\begin{proof}[of Theorem \ref{lawvere-lurie-algs}]
We write $\bAlg(\cC)$ for the full subcategory of functors $\Map_\Span(\cJ,\cCt)$ on the objects $f$ such that $qf=\cL$, and:
\begin{enumerate}[(i)]
 \item\label{d-equivs} For any set $A$, the canonical 1-cell of $\cJ$ defined by the constant maps $A_+:\Delta^0\rightarrow\Fins$ and $A:\Delta^0\star\Delta^0\rightarrow\Span$ is taken to an equivalence in $\cCt$,
 \item\label{d-lawvere} The restriction $f|_{\cJ_\Span}$ is a Lawvere algebra object.
 \item\label{d-lurie} The restriction $f|_{\cJ_\Fins}$ is a Lurie algebra object, in the sense that the following diagram factors with a dotted arrow, as shown, to give one:
\begin{displaymath}
 \xymatrix{\cJ_\Fins\ar[d]\ar[r]^\sim&\Fins\dar[r]&\cCot\ar[d]\\
           \cJ\ar[rr]&&\cCt.}
\end{displaymath}
\end{enumerate}

We notice that, in the presence of condition (\ref{d-equivs}), conditions (\ref{d-lawvere}) and (\ref{d-lurie}) are equivalent to one another. Indeed, both the sets of collapsing morphisms and the product properties match up under the given equivalences.

Also, condition (\ref{d-equivs}) is equivalent to saying that $f$ is a $q$-Kan extension of $f|_{\Span}$ along $\Span\rightarrow\cJ$ (as defined in \cite{HTT}*{Section 4.3.2}).

Indeed, to say that the following diagram is a Kan extension diagram
\begin{displaymath}
 \xymatrix{\cJ_\Span\ar[r]^{f|_{\cJ_\Span}}\ar[d]&\cCt\ar[d]^q\\
           \cJ\ar[r]\ar[ur]^f&\Span}
\end{displaymath}
is to say that, for every object $X\in\Ob(\cJ)$, the diagram
\begin{displaymath}
 \xymatrix{(\cJ_\Span)_{/X}\ar[r]^{f}\ar[d]&\cCt\ar[d]^q\\
           (\cJ_\Span)_{/X}\star 1\ar[r]\ar[ur]^f&\Span}
\end{displaymath}
makes $f(X)$ a $q$-colimit of $f_{/X}$. Here $(\cJ_\Span)_{/X}$ is notation for $(\cJ_{/X})\timeso{\cJ}(\cJ_{\Span})$.

If $X\in\Ob(\cJ_\Span)$, then $(\cJ_\Span)_{/X}\isom\Span_{/X}$, and the diagram is vacuously a $q$-colimit.

If $X_+\in\Ob(\cJ_\Fins)$, then $(\cJ_\Span)_{/X}\isom\Span_{/X}$, and the diagram is a $q$-colimit if and only if the 1-cell from $A$ to $A_+$ is taken to an equivalence in $\cCt$.

Since every map $f_0:\Span\rightarrow\cCt$ has $f_0\circ\cL$ as a $q$-left Kan extension to a map $\cJ\rightarrow\cCt$, \cite{HTT}*{Prop 4.3.2.15} says that the restriction functor $p:\bAlg(\cC)\rightarrow\Algt(\cC)$ is acyclic Kan.

Now, composition with $\cL$ defines a section of the functor $\bAlg(\cC)\rightarrow\Algt(\cC)$, and $\theta$ is the composition of this with the restriction map $p':\bAlg(\cC)\rightarrow\Algot(\cC)$.

Thus all we need to do is show that $p'$ is acyclic Kan, and \cite{HTT}*{Prop 4.3.2.15} says that this will follow from these two claims:
\begin{enumerate}[(a)]
\item\label{ak-a} Every $f_0\in\Algot(\cC)$ admits a $q$-Kan extension $f$, as shown fitting into the following diagram:
\begin{displaymath}
  \xymatrix{\Fins\ar[r]^{f_0}\ar[d]&\cCot\ar[r]&\cCt\ar[d]\\
            \cJ\ar[rr]\ar[urr]_f&&\Span}
\end{displaymath}
\item\label{ak-b} Given $f\in\sSet_\Span(\cJ,\cCt)$ such that $f_0=f|_{\Fins}$ is a Lurie algebra object, $f$ is a $q$-right Kan extension of $f_0$ if and only if $f$ satisfies condition (\ref{d-equivs}) above.
\end{enumerate}

To prove (\ref{ak-a}), for any object $K\in\Span$, we consider the quasicategory \[\cK_K=\Fins\timeso{\Span}(\Span_{K/}).\]
We write $g_K$ for the composite $\cK_K\rightarrow\Fins\rightarrow\cCot$.

%%% I think Lurie has a typo at this point, and attempts to map into $\cC$ instead of $\cCot$...

According to \cite{HTT}*{Lemma 4.3.2.13}, it will suffice to to show that, for every $K$, $g_K$ has a $q$-colimit in $\cCot$.

Since there is an injection $\Fins\rightarrow\Span$, there is an injection $\cK_K\rightarrow\Span_{K/}$; we thus write the objects of $\cK_K$ as morphisms $K\stackrel{a}{\leftarrow}Y\stackrel{b}{\rightarrow}Z$ of $\Span$. We let $\cK'_K$ denote the full subcategory on the objects where $b$ is an isomorphism and $a$ an injection.

As in the preceding proof of Theorem \ref{lawvere-lurie-mons}, we just need to show that $g'_K=g_K|_{\cK'_K}$ has a $q$-colimit in $\cCot$.

We write $\cK''_K$ for the full subcategory of $\cK'_K$ on the objects $K\leftarrow\{k\}\rightarrow\{k\}$. Because of the product property of monoidal objects, $g''_K=g'_K|_{\cK''_K}$ is a $q$-Kan extension of $g'_K$.

Thus, using \cite{HTT}*{Lemma 4.3.2.7}, we merely need to show that $g''_K$ has a $q$-limit in $\cCt$. But $f$ exhibits $f(K)$ as a $q$-limit of $g''_K$, and that proves~(\ref{ak-a}).

Similarly, this argument is reversible: $f$ is a $q$-right Kan extension of $f_0$ at $K_+$ if and only if $f$ induces an equivalence $f(K_+)\rightarrow f(K)$; this proves~(\ref{ak-b}).
\end{proof}

\end{document}